\documentclass{amsart}
\usepackage{amssymb,latexsym,amsmath,amsthm,amssymb}
\usepackage{enumitem}
\usepackage{nicefrac}
\usepackage{graphicx}
\newtheorem{theorem}{Theorem}[section]
\newtheorem{propo}[theorem]{Proposition}

\newtheorem{lemma}[theorem]{Lemma}
\theoremstyle{definition}
\newtheorem{remark}[theorem]{Remark}
\newtheorem{example}[theorem]{Example}
\newcommand{\lleft}{( \! (}
\newcommand{\rright}{) \! )}
\newcommand{\mik}{\preceq}
\newcommand{\al}{\alpha}
\newcommand{\lk}{\mbox{link}}

\newcommand{\abs}{\mbox{Abs}\,}
\newcommand{\dsp}{\displaystyle}
\newcommand{\sm}{\smallsetminus}

\newcommand{\os}{\caption{}}

\begin{document}
\title{The absolute order on the hyperoctahedral group}

\author{Myrto~Kallipoliti}

\address{Department of Mathematics
(Division of Algebra-Geometry)\\
University of Athens\\
Panepistimioupolis\\
15784 Athens, Greece}
\email{mirtok@math.uoa.gr}

\date{\today}

%
\begin{abstract}
The absolute order on the hyperoctahedral group $B_n$ is investigated.
It is proved that the order ideal of this poset generated by the Coxeter  elements  is homotopy Cohen-Macaulay and the 
M\"obius number of this ideal is computed. Moreover, it is
shown that every closed interval in the absolute order on $B_n$ is shellable and an example of a
non-Cohen-Macaulay interval in the absolute order on $D_4$ is given. Finally, the closed intervals
in the absolute order on $B_n$ and $D_n$ which are lattices are characterized and some of their
important enumerative invariants are computed.
\end{abstract}

\thanks{The present research will be part of the author's Doctoral Dissertation at the University of Athens}

\maketitle
\section{Introduction and results}
\label{intro}
Coxeter groups are fundamental combinatorial structures which appear in several areas of mathematics.
Partial orders on Coxeter groups often provide an important tool for understanding the questions of interest.
Examples of such partial orders are the Bruhat order and the weak order.
We refer the reader to \cite{bjo0, bb, Hu} for background  on Coxeter groups and their orderings.

In this work we study the absolute order. Let $W$ be a finite Coxeter group	 and let $\mathcal{T}$
be the set of \emph{all} reflections in $W$. The absolute order on $W$ is denoted by $\abs(W)$ and
defined as the partial order on $W$ whose Hasse diagram is obtained from the Cayley graph of $W$
with respect to $\mathcal{T}$ by directing its edges away from the identity (see Section \ref{abs}
for a precise definition). The poset $\abs(W)$ is locally self-dual and graded. It has a minimum
element, the identity $e\in W$, but will typically not have a maximum, since every Coxeter element
of $W$ is a maximal element of $\abs(W)$. Its rank function is called the absolute length and is
denoted by $\ell_\mathcal{T}$. The absolute length and order arise naturally in combinatorics
\cite{arm}, group theory \cite{Be, brwtt}, statistics \cite{Di} and invariant theory \cite{Hu}. For
instance, $\ell_{\mathcal{T}}(w)$ can also be defined as the codimension of the fixed space of $w$,
when $W$ acts faithfully as a group generated by orthogonal reflections on a vector space $V$ by
its standard geometric representation. Moreover, the rank generating polynomial of $\abs(W)$
satisfies
\[\sum_{w \in W}\ t^{\ell_{\mathcal{T}} (w)} \ = \ \prod_{i=1}^\ell\ (1 + e_i t), \]
where $e_1, e_2,\dots,e_\ell$ are the exponents \cite[Section 3.20]{Hu}
of $W$ and $\ell$ is its rank.
We refer to \cite[Section 2.4]{arm} and \cite[Section 1]{ca} for further discussion of
the importance of the absolute order and related historical remarks.

We will be interested in the combinatorics and topology of $\abs(W)$. These have been studied
extensively for the interval $[e,c]:=NC(W,c)$ of $\abs(W)$, known as the poset of noncrossing
partitions associated to $W$, where $c\in W$ denotes a Coxeter element. For instance, it was shown
in \cite{cbw} that $NC(W,c)$ is shellable for every finite Coxeter group $W$. In particular,
$NC(W,c)$ is Cohen-Macaulay over $\mathbb{Z}$ and the order complex of $NC(W,c)\sm\{e,c\}$ has the
homotopy type of a wedge of spheres. 

The problem to study the topology of the poset
$\abs(W)\sm\{e\}$ and to decide whether $\abs(W)$ is Cohen-Macaulay, or even shellable, was posed by
Reiner \cite[Problem 3.1] {arm2} and Athanasiadis (unpublished); see also \cite[Problem
3.3.7]{mwa}. Computer calculations carried out by Reiner showed that the absolute order is not
Cohen-Macaulay for the group $D_4$. This led Reiner to ask \cite[Problem 3.1] {arm2} whether the order ideal of $\abs(W)$ 
generated by the set of Coxeter elements is Cohen-Macaulay (or shellable) for every finite Coxeter group $W$. 
In the case of the symmetric group $S_n$ this ideal coincides with $\abs(S_n)$, since every maximal element of $S_n$ is a Coxeter element. 
Although it is not known whether $\abs(S_n)$ is  shellable, the following results were obtained in \cite{ca}.

\begin{theorem}\emph{(\cite[Theorem 1.1]{ca}).}
\label{thca1} The poset $\abs(S_n)$ is homotopy Cohen-Macaulay for every $n \ge 1$. In particular, the order
complex of $\abs(S_n)\sm\{e\}$ is homotopy equivalent to a wedge of $(n-2)$-dimensional spheres and
Cohen-Macaulay over $\mathbb{Z}$.
\end{theorem}

\begin{theorem}\emph{(\cite[Theorem 1.2]{ca}).}
\label{thca2}
Let $\bar{P}_n=\abs(S_n)\sm\{e\}$. The reduced Euler characteristic of the order complex $\Delta
(\bar{P}_n)$ satisfies
\[\sum_{n \ge 1} \ (-1)^n \, \tilde{\chi} (\Delta (\bar{P}_n)) \,
\frac{t^n}{n!} \ = \ 1 - C(t) \exp \left\{ -2t \, C(t) \right\},\]
where $C(t) = \frac{1}{2t} \, (1 - \sqrt{1-4t})$ is the ordinary generating function for the
Catalan numbers.
\end{theorem}

In the present paper we focus on the hyperoctahedral group $B_n$. 
We denote by $\mathcal{J}_n$ the order ideal of $\abs(B_n)$ generated by the Coxeter elements of $B_n$ 
and by $\bar{\mathcal{J}}_n$ its proper part $\mathcal{J}_n\sm \{e\}$. Contrary to the case of the symmetric
group, not every maximal element of $\abs(B_n)$ is a Coxeter element.
Our main results are as follows.

\begin{theorem}
\label{th3}
The poset $\mathcal{J}_n$ is homotopy Cohen-Macaulay for every $n\geq 2$. 
\end{theorem}

\begin{theorem}
\label{th3b}
The reduced Euler
characteristic of the order complex $\Delta(\bar{\mathcal{J}}_n)$ satisfies

\[\sum_{n \geq 2}  (-1)^n \tilde{\chi} (\Delta (\bar{\mathcal{J}}_n))
\frac{t^n}{n!}	= 1 - \sqrt{C(2t)} \exp \left\{ -2t	 C(2t) \right\}\left(1+\sum_{n \geq 1} 2^{n-1}{2n-1\choose n} \frac{t^n}{n}\right),\]
where $C(t) = \frac{1}{2t} \, (1 - \sqrt{1-4t})$ is the ordinary generating function for the
Catalan numbers.
\end{theorem}

The maximal (with respect to inclusion) intervals in $\abs(B_n)$ include the posets $NC^B(n)$ of
noncrossing partitions of type $B$, introduced and studied by Reiner  \cite{R}, and 
$NC^B(p,q)$ of annular noncrossing partitions, studied recently by Krattenthaler \cite{krat} and by Nica and Oancea \cite{N0}. 
We have the following result concerning the 
intervals of $\abs(B_n)$. 

\begin{theorem}
\label{th2}
Every interval of $\abs(B_n)$ is shellable.
\end{theorem}

Furthermore, we consider the absolute order on the group $D_n$ and give an example of a maximal
element $x$ of $\abs(D_4)$ for which the interval $[e,x]$ is not Cohen-Macaulay over any field
(Remark \ref{exd4}). This is in accordance with Reiner's computations, showing that $\abs(D_4)$ is
not Cohen-Macaulay and answers in the negative a question raised by Athanasiadis (personal
communication), asking whether all intervals in the absolute order on  Coxeter groups are shellable. 
Moreover, it shows that $\abs(D_n)$ is not Cohen-Macaulay over any field for every $n\geq 4$. 
It is an open problem to decide whether $\abs(B_n)$ is Cohen-Macaulay for every $n\geq 2$ and whether 
the order ideal of $\abs(W)$ generated by the set of
Coxeter elements is Cohen-Macaulay for every Coxeter group $W$ \cite[Problem 3.1]{arm2}.

This paper is organized as follows.
In Section \ref{prepre} we fix notation and terminology related
to partially ordered sets and simplicial complexes and discuss the absolute order on
the classical finite reflection groups. 
In Section \ref{elel} we prove Theorem \ref{th2} by showing that every closed interval of
$\abs(B_n)$ admits an EL-labeling. Theorems \ref{th3} and \ref{th3b} are proved in Section \ref{cmcm}. Our method to establish
homotopy Cohen-Macaulayness is different from that of \cite{ca}. It is based on a poset fiber
theorem due to Quillen \cite[Corollary 9.7]{Q}. 
The same method gives an alternative proof of Theorem \ref{thca1}, 
which is also included in Section \ref{cmcm}. In Section \ref{latlat} we characterize the closed intervals in $\abs(B_n)$ and $\abs(D_n)$ 
which are lattices. In Section \ref{spespe} we study a special case of
such an interval, namely the maximal interval $[e,x]$ of $\abs(B_n)$, where $x=t_1t_2\cdots t_n$
and each $t_i$ is a balanced reflection. Finally, in Section \ref{telostelos}
we compute the zeta polynomial, cardinality and M\"obius function of the intervals of $\abs(B_n)$
which are lattices. These computations are based on results of Goulden, Nica and Oancea \cite{N}
concerning enumerative properties of the poset $NC^B(n-1,1)$.


\section{Preliminaries}
\label{prepre}
\subsection{Partial orders and simplicial complexes}
Let $(P,\leq)$ be a finite partially ordered set (poset for short) and $x,y\in P$. We say that $y$
\emph{covers} $x$, and write $x\to y$, if $x<y$ and there is
no $z\in P$ such that $x< z< y$. The poset $P$ is called \emph{bounded} if there exist elements
$\hat{0}$ and $\hat{1}$ such that $\hat{0}\leq x\leq \hat{1}$ for every $x\in P$. 
The elements of $P$ which cover $\hat{0}$ are called \emph{atoms}. 
A subset $C$ of a poset $P$
is called a \emph{chain} if any two elements of $C$ are comparable in $P$. The length of a (finite)
chain $C$ is equal to
$|C|-1$. We say that $P$ is \emph{graded} if all maximal chains of $P$ have the same length.  
In that case, the common length of all maximal chains of $P$ is called \emph{rank}. Moreover, assuming $P$ has a $\hat{0}$ element, 
there exists a unique function $\rho:P\to \mathbb{N}$, called the \emph{rank function} of $P$, such that
\[\rho(y)=\left\{
\begin{array}{ll}
0 & \mbox{if $y=\hat{0}$}, \\
\rho(x)+1 & \mbox{if $x\to y$}.
\end{array}
\right.\] 
We say that $x$ has \emph{rank} $i$ if
$\rho(x)=i$. 
For $x\leq y$ in $P$ we denote by $[x,y]$ the closed interval $\{z \in P : x \leq z \leq y\}$ of $P$, endowed with the partial order induced from $P$. 
If $S$ is a subset of $P$, then the \emph{order ideal} of $P$ generated by
$S$ is the subposet $\langle S\rangle$ of $P$ consisting of all $x\in P$ for which $x\mik y$ holds for some $y\in S$.
We will write $\langle y_1,y_2,\dots,y_m \rangle$ for the order ideal of $P$ generated by the set
$\{y_1,y_2,\dots,y_m\}$.
Given two posets $(P,\leq_P)$ and $(Q,\leq_Q)$, a map $f:P\to Q$ is called a
\emph{poset map} if it is order preserving, i.e. $x\leq_Py$ implies $f(x)\leq_Qf(y)$ for all
$x,y\in P$. If, in addition, $f$ is a bijection with order preserving inverse, then $f$ is said to be a 
\emph{poset isomorphism}.  
The posets $P$
and $Q$ are said to be \emph{isomorphic}, and we write $P\cong Q$, if there
exists a  poset isomorphism $f: P \to Q$. Assuming that $P$ and $Q$ are graded, 
the map $f:P\to Q$ is called \emph{rank-preserving} if for every $x\in P$, 
the rank of $f(x)$ in $Q$ is equal to the rank of $x$ in $P$.  
The \emph{direct product} of
$P$ and $Q$ is the poset $P\times Q$ on the set $\{(x,y):x\in P,\, \ y\in Q\}$
for which $(x,y)\leq (x',y')$ holds in $P\times Q$ if $x\leq_P x'$ and $y\leq_Q y'$.
The \emph{dual} of $P$ is the poset $P^*$ defined on the same ground
set as $P$ by letting $x\leq y$ in $P^*$ if and only if $y\leq x$ in $P$. The poset $P$ is called \emph{self-dual} if $P$ and $P^*$ are isomorphic 
and \emph{locally self-dual} if every closed interval of $P$ is self-dual. 
For more information on partially ordered sets we refer the reader to \cite[Chapter 3]{St}.

We recall the notion of EL-shellability, defined by Bj\"orner \cite{bjo1}.
Assume that $P$ is bounded and graded and let
$C(P)=\{(a,b)\in P\times P:\ a\to b\}$ be the set of covering relations of $P$.
An \emph{edge-labeling} of $P$ is a map $\lambda:C(P)\to \Lambda$, where $\Lambda$ is some poset.
Let $[x,y]$ be a closed interval of $P$ of rank $n$.  
To each maximal chain $c:\,x\to x_1\to\cdots\to x_{n-1}\to y$ of $[x, y]$ we associate the sequence
$\lambda(c)=(\lambda(x,x_1),\lambda(x_1,x_2),\dots, \lambda(x_{n-1},y)\,)$. 
We say that $c$ is \emph{strictly increasing}
if the sequence $\lambda(c)$ is strictly increasing in the order of $\Lambda$.
The maximal chains of $[x, y]$ can be totally ordered by using the lexicographic order on the
corresponding sequences. An \emph{edge-lexicographic labeling (EL-
labeling)} of $P$ is an edge labeling such that in each closed interval $[x, y]$
of $P$ there is a unique strictly increasing maximal chain and this chain lexicographically precedes
all other maximal chains of $[x, y]$. The poset $P$ is called \emph{EL-shellable} if it admits an EL-labeling. 
A finite poset $P$ of rank $d$ with a minimum element is called \emph{strongly constructible} \cite{ca}
if it is bounded and pure shellable, or it can be
written as a union $P = I_1 \cup I_2$ of two strongly
constructible proper ideals $I_1, I_2$ of rank $n$, such that $I_1
\cap I_2$ is strongly constructible of rank at least $n-1$. 


Let $V$ be a nonempty finite set.
An \emph{abstract simplicial complex} $\Delta$ on the vertex set $V$ is a collection
of subsets of $V$ such that $\{v\} \in\Delta$ for every $v\in V$ and such that $G \in \Delta$ and $F\subseteq G$ imply $F\in\Delta$.
The elements of $V$ and $\Delta$ are called \emph{vertices} and \emph{faces} of $\Delta$, respectively.
The maximal faces are called \emph{facets}.
The dimension of a face $F\in \Delta$ is equal to $|F|-1$ and is denoted by $\dim F$.
The \emph{dimension} of $\Delta$ is defined as the maximum dimension of a face of $\Delta$ and is denoted by $\dim\Delta$.
If all facets of $\Delta$ have the same dimension, then $\Delta$ is said to be \emph{pure}.
The \emph{link} of a face $F$ of a simplicial complex
$\Delta$ is defined as $\lk_{\Delta}(F)=\{G\smallsetminus F:\,G\in\Delta,F\subseteq G\}$.
All topological properties of an abstract simplicial complex $\Delta$ we mention will refer to those of its geometric
realization $\|\Delta\|$. The complex $\Delta$ is said to be
\emph{homotopy Cohen-Macaulay} if for all $F\in\Delta$ the link of $F$ is topologically $(\dim
\lk_\Delta$$(F)-1)$-connected. 
For a facet $G$ of a simplicial complex $\Delta$, we denote by $\bar{G}$ the Boolean interval $[\varnothing,G]$.
A pure $d$-dimensional simplicial complex $\Delta$ is 
\emph{shellable} if there exists a total ordering $G_1, G_2,\dots,G_m$ of
the set of facets of $\Delta$ such that for all $1 < i \le m$, the
intersection of $\bar{G}_1\cup \bar{G}_2 \cup \, \cdots \, \cup \bar{G}_{i-1}$ with $\bar{G}_i$ is pure
of dimension $d-1$. 
For a $d$-dimensional simplicial complex we have the following implications:
pure shellable $\Rightarrow$  homotopy Cohen-Macaulay $\Rightarrow$
homotopy equivalent to a wedge of $d$-dimensional spheres. For background concerning the topology of simplicial complexes we refer to
\cite{bjo2} and \cite{mwa}.

To every poset $P$ we associate an abstract simplicial complex $\Delta(P)$, called the \emph{order
complex} of $P$. The vertices of $\Delta(P)$ are the elements of $P$ and its faces are the chains
of $P$. If $P$ is graded of rank $n$, then $\Delta(P)$ is pure of dimension $n$. All topological
properties of a poset $P$ we mention will refer to those of the geometric realization of
$\Delta(P)$. We say that a poset $P$ is \emph{shellable} if its order complex $\Delta(P)$ is
shellable and recall that every EL-shellable poset is shellable \cite[Theorem 2.3]{bjo1}.
We also recall the following lemmas.

\begin{lemma}
\label{strcon}
Let $P$ and $Q$ be finite posets, each with a minimum element. 
\begin{enumerate}[leftmargin=*, itemsep=5pt]
\item[\emph{(i)}] \emph{\cite[Lemma 3.7]{ca}} 
If $P$ and $Q$ are strongly constructible, then so is
	the direct product $P\times Q$.
\item [\emph{(ii)}] \emph{\cite[Lemma 3.8]{ca}} If $P$ is the union
	of strongly constructible ideals $I_1, I_2,\dots,I_k$  of $P$ of rank $n$ and the
	intersection of any two or more of these ideals is strongly constructible of rank $n$ or
	$n-1$, then $P$ is also strongly constructible.
\end{enumerate}
\end{lemma}

\begin{lemma}
Every strongly constructible poset is homotopy Cohen-Macaulay.
\end{lemma}

\begin{proof}
It follows from \cite[Proposition 3.6]{ca} and \cite[Corollary 3.3]{ca}.
\end{proof}

\begin{lemma}
\label{tomes}
Let $P$ and $Q$ be finite posets, each with a minimum element. 
\begin{enumerate}[leftmargin=*, itemsep=5pt]
\item[\emph{(i)}] If $P$ and $Q$ are homotopy Cohen-Macaulay, then so is
	the direct product $P\times Q$.
\item [\emph{(ii)}] If $P$ is the union
	of homotopy Cohen-Macaulay ideals $I_1, I_2,\dots,I_k$ of $P$ of rank $n$ and the
	intersection of any two or more of these ideals is homotopy Cohen-Macaulay of rank $n$ or
	$n-1$, then $P$ is also homotopy Cohen-Macaulay.
\end{enumerate}
\end{lemma}

\begin{proof} 
The first part follows from \cite[Corollary 3.8]{bww2}. The proof of
the second part is similar to that of \cite[Lemma 3.4]{ca}.
\end{proof}

\subsection{The absolute length and absolute order}
\label{abs}
Let $W$ be a finite Coxeter group and let $\mathcal{T}$ denote the set of all reflections in $W$.
Given $w \in W$, the \emph{absolute length} of $w$ is defined as  the smallest integer $k$ such that $w$ can be
written as a product of $k$ elements of $\mathcal{T}$; it is denoted by $\ell_{\mathcal{T}} (w)$. 
The \emph{absolute order} $\abs(W)$ is the partial order  $\mik$ on $W$ defined by
\[ u \mik v \ \ \ \mbox{if and only if} \ \ \ \ell_{\mathcal{T}}(u) \, + \,
\ell_{\mathcal{T}}(u^{-1} v) \, = \, \ell_{\mathcal{T}} (v) \]
for $u, v \in W$. Equivalently, $\preceq$ is the partial order on $W$ with covering relations $w
\to wt$, where $w \in W$ and $t \in\mathcal{T}$ are such that $\ell_{\mathcal{T}} (w) <
\ell_{\mathcal{T}} (wt)$. In that case we write $w\stackrel{t}{\to}wt$. The poset $\abs(W)$ is
graded with rank function $\ell_{\mathcal{T}}$. 

Every closed interval in $W$ is isomorphic to one which contains 
the identity. Specifically, we have the following lemma (see also \cite[Lemma 3.7]{cbw}). 

\begin{lemma}
\label{lemarm}
Let $u,v\in  W$ with $u\mik v$. 
The map $\phi: [u,v]  \to [e, u^{-1}v]$ defined by $\phi(w)= u^{-1}w$ is a poset isomorphism.
\end{lemma}

\begin{proof}
It follows from \cite[Lemma 2.5.4]{arm} by an argument similar to that in the proof of \cite[Proposition 2.6.11]{arm}. 
\end{proof}
For more information on the absolute order on $W$ we refer the reader to \cite[Section 2.4]{arm}.
 
\subsection*{The absolute order on $S_n$}
We view the group $S_n$ as the group of permutations of the set $\{1,2,\dots, n\}$. The set
$\mathcal{T}$ of reflections of $S_n$ is equal to the set of all transpositions $(i\,j)$, where
$1\leq i<j\leq n$. The length $\ell_{\mathcal{T}}(w)$ of $w\in S_n$ is equal to $n-\gamma(w)$,
where $\gamma(w)$ denotes the number of cycles in the cycle decomposition of $w$. Given a cycle $c =
(i_1\, i_2\, \cdots\, i_r)$ in $S_n$ and indices $1\leq j_1<j_2<\cdots<j_s\leq r$, we say that the
cycle $(i_{j_1}\,i_{j_2}\,\cdots\,i_{j_s})$ can be obtained from $c$ by deleting elements.
Given two disjoint cycles $a, b$ in $S_n$ each of which can be obtained from $c$ by deleting
elements, we say that $a$ and $b$ are noncrossing with respect to $c$ if there does not exist a
cycle $(i\, j\, k\, l)$ of length four which can be obtained from $c$ by deleting elements, such
that $i, k$ are elements of $a$ and $j, l$ are elements of $b$. For instance, if $n = 9$ and $c =
(3\, 5\, 1\, 9\, 2\, 6\, 4)$ then the cycles $(3\, 6\, 4)$ and $(5\, 9\, 2)$ are noncrossing with
respect to $c$ but $(3\, 2\, 4)$ and $(5\, 9\, 6)$ are not. It can be verified \cite[Section 2]{Br}
that for $u, v\in S_n$ we have $u \mik v$ if and only if 
\begin{itemize}
\item every cycle in the cycle decomposition for $u$ can be obtained from some cycle in the cycle
decomposition for $v$ by deleting elements and 
\item any two cycles of $u$ which can be obtained
from the same cycle $c$ of $v$ by deleting elements are noncrossing with respect to $c$.
\end{itemize} 
Clearly, the maximal elements of $\abs(S_n)$ are precisely the $n$-cycles, which are the Coxeter elements of $S_n$.  
Figure \ref{s3} illustrates the Hasse diagram of the poset $\abs(S_3)$.
\begin{figure}[h]
\begin{center}
\includegraphics[width=2.7in]{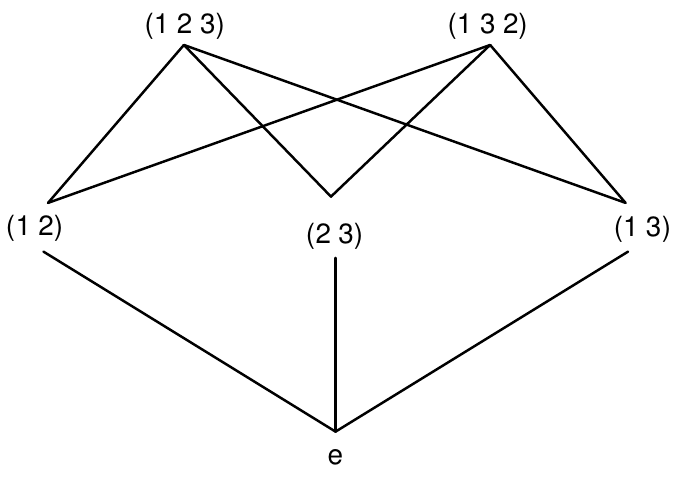}
\end{center}
\os
\label{s3}
\end{figure}

\subsection*{The absolute order on $B_n$}
\label{relationbn}
We view the hyperoctahedral group $B_n$ as the group of permutations
$w$ of the set $\{\pm1,\pm2,\dots,\pm n\}$ satisfying $w(-i)=-w(i)$
for $1\leq i\leq n$.
Following \cite{brwtt}, the permutation which has cycle form
$(a_1\,a_2\,\cdots\,a_k)(-a_1\,-a_2\,\cdots\,-a_k)$ is denoted by
$\lleft a_1,a_2,\dots, a_k\rright$ and is called a \emph{paired} $k$-cycle,
while the cycle $(a_1\,a_2\,\cdots\, a_k\,-a_1\,-a_2\,\cdots\,-a_k)$
is denoted by $[a_1,a_2,\dots,a_k]$ and is called a \emph{balanced} $k$-cycle.
Every element $w \in B_n$ can be written as a product of disjoint paired or balanced cycles, called cycles of $w$.
With this notation, the set $\mathcal{T}$ of reflections of $B_n$ is equal to the union
\begin{equation}
\label{T_B}
\{[i]:\,1\leq i\leq n\}\,\cup\,\{\lleft i,j\rright,\lleft i,-j\rright:\,1\leq i<j\leq n\}.
\end{equation}
The length $\ell_{\mathcal{T}}(w)$ of $w\in B_n$ is equal to $n-\gamma(w)$, where $\gamma(w)$
denotes the number of paired cycles in the cycle decomposition of $w$. An element $w\in B_n$ is
maximal in $\abs(B_n)$ if and only if it can be written as a product of disjoint balanced cycles
whose lengths sum to $n$. The Coxeter elements of $B_n$ are precisely the balanced $n$-cycles. 
The covering relations $w\stackrel{t}{\to}wt$ of $\abs(B_n)$, when $w$ and
$t$ are non-disjoint cycles, can be described as follows: For	$1\leq i<j\leq m\leq n$, we have:

\

\begin{enumerate}
\item[(a)]\label{a1}$\lleft a_1,\dots, a_{i-1},a_{i+1},\dots,a_m\rright\stackrel{\lleft a_{i-1},a_i\rright}{\longrightarrow}\lleft a_1,\dots,a_m\rright$

\item[(b)]\label{a2}$\lleft a_1,\dots,a_m\rright\stackrel{[a_i]}{\longrightarrow}[a_1,\dots,a_{i-1},a_i,-a_{i+1},\dots,-a_m]$

\item[(c)]\label{a3}$\lleft a_1,\dots,a_m\rright\stackrel{\lleft a_i,-a_j\rright}{\longrightarrow}[a_1,\dots,a_i,-a_{j+1},\dots,-a_m][a_{i+1},\dots,a_j]$

\item[(d)]\label{a4}$[a_1,\dots, a_{i-1},a_{i+1},\dots,a_m]\stackrel{\lleft a_{i-1},a_i\rright}{\longrightarrow}[a_1,\dots,a_m]$

\item[(e)]\label{a5}$[a_1,\dots, a_j]\lleft a_{j+1},\dots,a_m\rright\stackrel{\lleft a_j,a_m\rright}{\longrightarrow}[a_1,\dots,a_m]$

\item[(f)]\label{a6}$\lleft a_1,\dots, a_j\rright\lleft a_{j+1},\dots,a_m\rright\stackrel{\lleft a_j,a_m\rright}{\longrightarrow}\lleft a_1,\dots,a_m\rright$
\end{enumerate}

\

\noindent where $a_1,\dots,a_m$ are elements of  $\{\pm1,\dots,\pm n\}$ with pairwise distinct
absolute values. Figure \ref{b2} illustrates the Hasse diagram of the poset $\abs(B_2)$.

\begin{remark}
\label{anal}
Let $w=bp$ be an element in $B_n$, where $b$ (respectively,  $p$) is the product of all balanced  (respectively, paired) cycles of $w$.
The covering relations of $\abs(B_n)$ imply the poset isomorphism $[e,w]\cong [e,b]\times[e,p]$.
Moreover, if $p=p_1\cdots p_k$ is written as a product of disjoint paired cycles, then 
\[[e,w]\cong [e,b]\times[e,p_1]\times\cdots\times[e,p_k].\] 
\end{remark}

\begin{figure}[h]
\begin{center}
\includegraphics[width=2.7in]{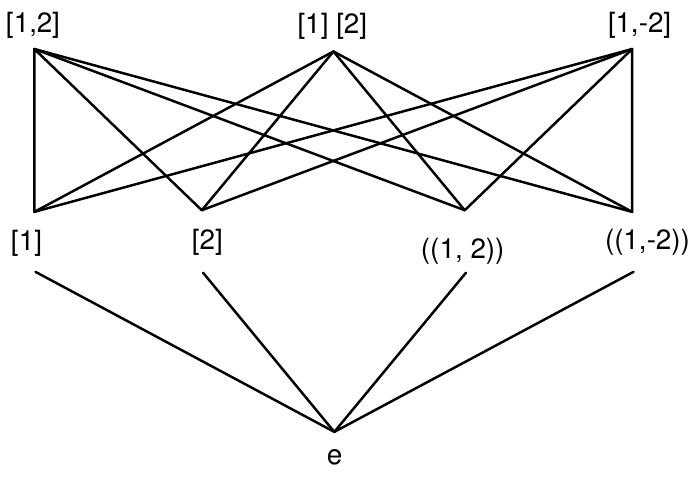}
\end{center}
\os
\label{b2}
\end{figure}

\subsection*{The absolute order on $D_n$}
\label{dn}
The Coxeter group $D_n$ is the subgroup of index two of the group $B_n$,
generated by the set of reflections
\begin{equation}
\label{T_D}
\{\lleft i,j\rright,\lleft i,-j\rright:1\leq i<j\leq n\}
\end{equation}
(these are all reflections in $D_n$). An element $w\in B_n$ belongs to $D_n$ if and only if $w$ has
an even number of balanced cycles in its cycle decomposition. The absolute length on $D_n$
is the restriction of the absolute length of $B_n$ on the set $D_n$ and hence $\abs(D_n)$ is a
subposet of $\abs(B_n)$. Every
Coxeter element of $D_n$ has the form $[a_1,a_2,\dots,a_{n-1}][a_n]$, where
$a_1,\dots,a_n$ are elements of $\{\pm1,\dots,\pm n\}$ with pairwise distinct absolute values.

\subsection*{Projections}
We recall that $\mathcal{J}_n$ denotes the order ideal of $\abs(B_n)$ generated by the Coxeter elements of $B_n$. 
Let $P_n$ be $\abs(S_n)$ or $\mathcal{J}_n$ for some  $n\geq 2$.  
For $i\in\{1,2,\dots,n\}$ we define a map $\pi_i:P_n\to P_n$ by letting $\pi_i(w)$ be the permutation obtained when 
 $\pm i$ is deleted from the cycle decomposition of $w$. For example, if $n=i=5$ and
$w=[1,-5,2]\lleft 3,-4\rright\in \mathcal{J}_5$, then $\pi_i(w)=[1,2]\lleft 3,-4\rright$. 

\begin{lemma}
\label{proj1}The following hold for the map $\pi_i:P_n\to P_n$.
\begin{enumerate}[leftmargin=*, itemsep=5pt]
\item [\emph{(i)}] $\pi_i(w)\mik w$ for every $w\in P_n$. 
\item [\emph{(ii)}] $\pi_i$ is a poset map.
\end{enumerate}
\end{lemma}

\begin{proof}
Let $w\in P_n$. If $w(i)=i$, then clearly $\pi_i(w)=w$. Suppose that $w(i)\neq i$. Then it follows from our description of $\abs(S_n)$ and from 
the covering relations of types (a) and (d) of $\abs(B_n)$, that $\pi_i(w)$ is covered by $w$.  Hence $\pi_i(w)\mik w$. 
This proves (i). To prove (ii), it suffices to show that for every covering relation $u\to v$ in $P_n$ we have either $\pi_i(u)=\pi_i(v)$ or $\pi_i(u)\to\pi_i(v)$. 
Again, this follows from our discussion of $\abs(S_n)$ and from our list of covering relations of $\abs(B_n)$.  
\end{proof}

\begin{lemma}
\label{antenatel}
Let $P_n$ stand for  either $\abs(S_n)$ for every $n\geq 1$, or $\mathcal{J}_n$ for every $n\geq 2$. 
Let also $w\in P_n$ and $u\in P_{n-1}$ be such that   $\pi_n(w)\mik u$. 
Then there exists an element $v\in P_n$ which covers $u$ and satisfies  $\pi_n(v)=u$ and $w\mik v$.
\end{lemma}

\begin{proof}We may assume that  $w$ does not fix $n$, since otherwise the result is trivial. 
Suppose that $\pi_n(w)=w_1\cdots w_l$ and $u=u_1\cdots u_r$ are written  as  products of disjoint cycles in $P_{n-1}$. 

\smallskip

\noindent{\bf~Case 1:} $P_n=\abs(S_n)$ for $n\geq 1$. 
Then there is an index $i\in\{1,2,\dots, l\}$ such that
$w$ is obtained from $\pi_n(w)$ by inserting $n$ in the cycle $w_i$. 
Let $y$ be the cycle of $w$ containing $n$, so that $\pi_n(y)=w_i$. From the description of the absolute order on $S_n$ given in this section, it follows that $w_i\mik u_j$ for some  
$j\in\{1,2,\dots,r\}$. 
We may insert $n$ in the cycle $u_j$ so that the resulting cycle $v_j$ satisfies $y\mik v_j$. 
 Let $v$ be the element of $S_n$ obtained by
replacing $u_j$ in the cycle decomposition of $u$ by $v_j$. Then $u$ is covered by $v,\,\pi_n(v)=u$ and $w\mik v$.

\smallskip

\noindent{\bf~Case 2:}  $P_n=\mathcal{J}_n$ for $n\geq 2$. 
The result follows by a simple modification of the argument in the previous case, if $[n]$ is not a cycle of $w$. 
Assume the contrary, so that $w=\pi_n(w)[n]$ and all cycles of $\pi_n(w)$ are paired. 
If $u$ has no balanced cycle, then $w\mik u[n]\in \mathcal{J}_n$ and hence $v=u[n]$ has the desired properties. 
Suppose that $u$ has a balanced cycle in its cycle decomposition, say $b=[a_1,\dots,a_k]$. 
We denote by $p$ the product of all paired cycles of $u$, so that $u=bp$. 
If $\pi_n(w)\mik p$, then the choice $v=[a_1,\dots,a_k,n]p$ works. 
Otherwise, we may assume that there is an index $m\in \{1,2,\dots,l\}$ such that $w_1\cdots w_m\mik b$ and $w_i$ and $b$ are disjoint for every $i>m$.  
From the covering relations of $\abs(B_n)$ of types (a), (b) and (f) it follows that there is a paired cycle $c$ which is covered by $b$ and satisfies $w_1\cdots w_m\mik c$. Thus $\pi_n(w)\mik cp\mik u$. 
More specifically, $c$ has the form $\lleft a_1,\dots, a_i,-a_{i+1},\dots,-a_k\rright$ for some $i\in \{2,\dots, k\}$. 
We set $v=[a_1,\dots,a_i,n,a_{i+1},\dots,a_l]p$. Then $v$ covers $u$ and $w\mik cp[n]\mik v$.  This concludes the proof of the lemma. 
\end{proof}


\section{Shellability}
\label{elel}
In this section we prove Theorem \ref{th2} by showing that every closed interval of $\abs(B_n)$
admits an EL-labeling.
Let $C(B_n)$ 
be the set of covering relations of	 $\abs(B_n)$
and $(a,b)\in C(B_n)$. Then $a^{-1}b$ is a reflection of $B_n$, thus
either $a^{-1}b=[i]$ for some $i\in\{1,2,\dots,n\}$, or there exist $i,j\in\{1,2,\dots,n\}$, with $i<j$, such that
$a^{-1}b=\lleft i,j\rright$ or $a^{-1}b=\lleft i,-j\rright$.
We define a map $\lambda:C(B_n)\to \{1,2,\dots,n\}$ as follows:
\[\lambda(a,b)=\left\{
\begin{array}{ll}

i & \mbox{if $a^{-1}b=[i]$}, \\
j & \mbox{if $a^{-1}b=\lleft i,j\rright$ or $\lleft i,-j\rright$}.

\end{array}
\right.\]

A similar labeling was used by Biane \cite{Bi} in order to study the maximal chains
of the poset $NC^B(n)$ of noncrossing $B_n$-partitions.
Figure \ref{figpxel} illustrates the Hasse diagram of the interval
$\left[e,[3,-4]\lleft 1,2\rright\right]$, together with the corresponding labels.

\

\begin{figure}[h]
\begin{center}
\includegraphics[width=4.5in]{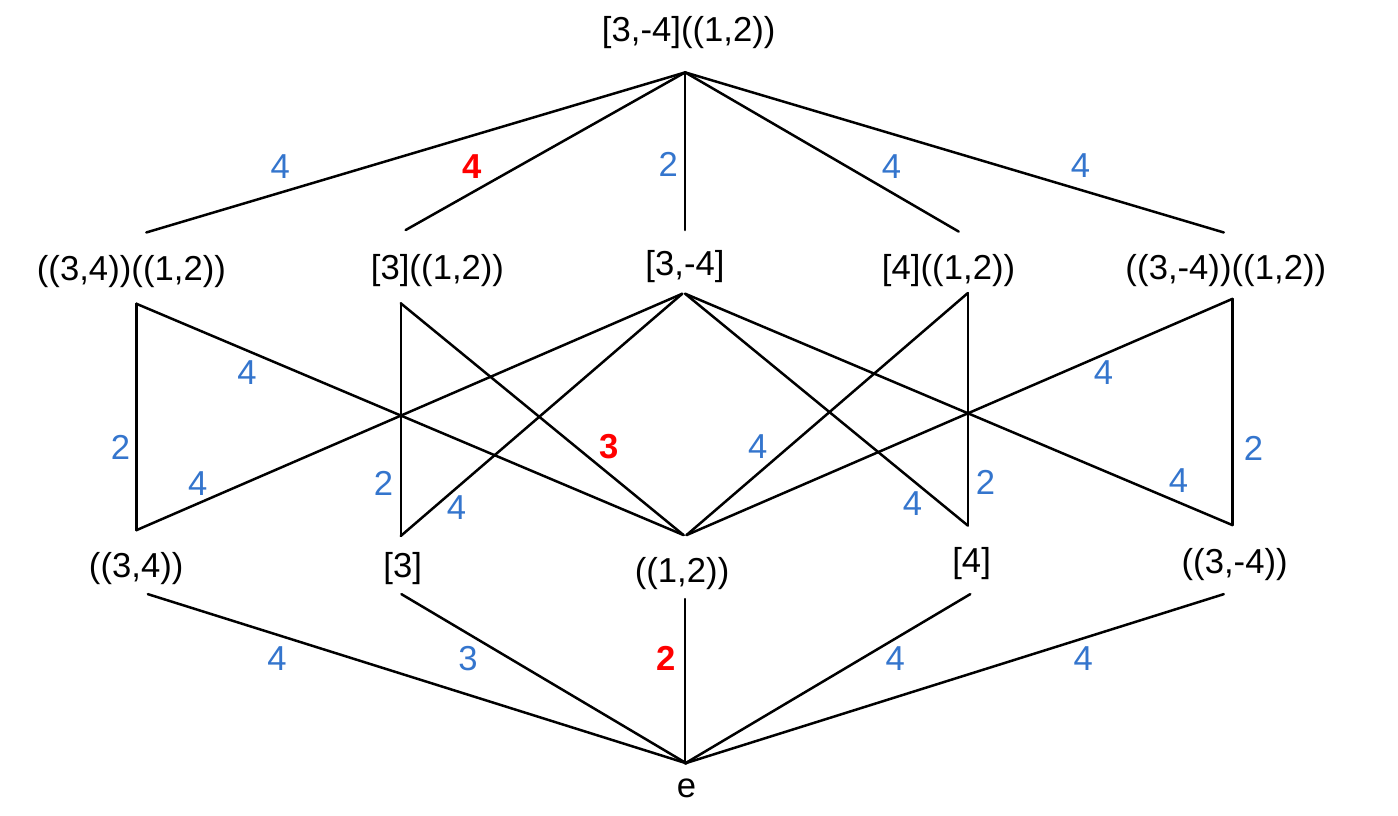}
\end{center}
\os
\label{figpxel}
\end{figure}

\begin{propo}
\label{el}
Let $u,v\in B_n$ with $u\mik v$. Then, the restriction of the map $\lambda$ to the interval $[u,v]$ is an EL-labeling. 
\end{propo}

\begin{proof}
Let $u,v\in B_n$ with $u\mik v$. We consider the poset isomorphism $\phi:[u,v]\to[e,u^{-1}v]$ from Lemma \ref{lemarm}.
Let $(a,b)\in C([u,v])$. 
Then we have $\phi(a)^{-1}\phi(b)=(u^{-1}a)^{-1}u^{-1}b=a^{-1}uu^{-1}b=a^{-1}b,$
which implies that $\lambda(a,b)=\lambda(\phi(a),\phi(b))$.
Thus, it suffices to show that $\lambda|_{[e,w]}$ is an EL-labeling for the interval $[e,w]$, where $w=u^{-1}v$.

Let $b_1b_2\cdots b_k\, p_1p_2\cdots p_l$ be the cycle decomposition of $w$, where
$b_i=[b_i^1,\dots,b_i^{k_i}]$ for $i\leq k$ and $p_j=\lleft p_j^1,\dots,p_j^{l_j}\rright$, with
$p_j^1=\min\{|p_j^m|:1\leq m\leq l_j\}$ for $j\leq l$.
We consider the sequence of positive integers obtained by placing the numbers $|b_i^h|$ and
$|p_j^m|$, for $i,j,h\geq 1$ and $m>1$, in increasing order. There are $r=\ell_{\mathcal{T}}(w)$ such integers.
To simplify the notation, we denote by $c(w)=(c_1,c_2,\dots,c_r)$ this sequence and say that
$c_{\mu}$ ($\mu=1,2,\dots,r$) belongs to a balanced (respectively, paired) cycle if it is equal to
some $|b_i^h|$ (respectively, $|p_j^m|$). Clearly, we have
\begin{equation}
\label{al}
c_1<c_2<\dots<c_r
\end{equation}
and $\lambda(a,b)\in\{c_1,c_2,\dots,c_r\}$ for all $(a,b)\in C([e,w])$.
To the sequence (\ref{al})
corresponds a unique maximal chain
\[\mathcal{C}_w:\ w_0=e\stackrel{c_1}{\to} w_1\stackrel{c_2}{\to} w_2\stackrel{c_3}{\to}\dots\stackrel{c_r}{\to} w_r=w,\]
which can be constructed inductively as follows (here, the integer $\kappa$ in
$a\stackrel{\kappa}{\to}b$ denotes the label $\lambda(a,b)$). If $c_1$ belongs to a balanced cycle,
then $w_1=[c_1]$. Otherwise, $c_1$ belongs to some $p_i$, say $p_1$, and we set $w_1$ to be either
$\lleft p_1^1, c_1\rright$ or $\lleft p_1^1,-c_1\rright$, so that $w_1\mik p_1$ holds. Note that in
both cases we have $\lambda(e,w_1)=c_1$ and $\lambda(e,w_1)<\lambda(e,w)$ for any other atom
$t\in[e,w]$. Indeed, suppose that there is an atom $t\neq w_1$ such that $\lambda(e,t)=c_1$. We
assume first that $c_1$ belongs to a balanced cycle, so $w_1=[c_1]$. Then $t$ is a reflection of
the form $\lleft c_0,\pm c_1\rright$, where $c_0<c_1$ and, therefore, $c_0$ belongs to some paired
cycle of $w$ (if not then $c_1$ would not be minimum). However from the covering relations of $\abs(B_n)$ written
at the end of Section \ref{relationbn} it follows that $\lleft c_0,\pm c_1\rright\not\mik w$, thus	
$\lleft c_0,\pm c_1\rright\not\in[e,w]$, a contradiction.
Therefore $c_1$ belongs to a paired cycle of $w$, say $p_1$, and $w_1,t$ are both paired reflections. 
Without loss of generality, let $w_1=\lleft p_1^1,c_1\rright$ and $t=\lleft c_0,c_1\rright$, for
some  $c_0<c_1$. By the first covering relation written at the end of Section \ref{relationbn} and
the definition of $\lambda$, it follows that $c_0=p_1^1$, thus $w_1=t$, again a contradiction.

Suppose now that we have uniquely defined the elements $w_1,w_2,\dots, w_j$, so that for every
$i=1,2,\dots,j$ we have $w_{i-1}\to w_i$ with $\lambda(w_{i-1},w_i)=c_i$	 and
$\lambda(w_{i-1},w_i)<\lambda(w_{i-1},z)$ for every $z\in[e,w]$ such that $z\neq w_i$ and
$w_{i-1}\to z$. We consider the number $c_{j+1}$ and distinguish two cases.

\smallskip

\noindent{\bf~Case 1:} $c_{j+1}$ belongs to a cycle whose elements have not been used.
In this case, if $c_{j+1}$ belongs to a balanced cycle, then we set $w_{j+1}=w_j[c_{j+1}]$,
while if $c_{j+1}$ belongs to $p_s$ for some $s\in\{1,2,\dots, l\}$, then we set
$w_{j+1}$ to be either $w_j\,\lleft p_s^1, c_{j+1}\rright$ or $w_j\,\lleft p_s^1,-c_{j+1}\rright$,
so that $w_j^{-1}w_{j+1}\mik p_s$ holds.

\smallskip

\noindent{\bf~Case 2:} $c_{j+1}$ belongs to a cycle some element of which has been used. Then there
exists an index $i<j+1$ such that $c_i$ belongs to the same cycle as $c_{j+1}$. If $c_i,c_{j+1}$
belong to some $b_s$, then there is a balanced cycle of $w_j$, say $a$, that contains $c_i$. In
this case we set $w_{j+1}$ to be the permutation that we obtain from $w_j$ if we add the number
$c_{j+1}$ in the cycle $a$ in the same order and with the same sign that it appears in $b_s$. We
proceed similarly if $c_i,c_{j+1}$ belong to the same paired cycle.

\smallskip

In both cases we have $\lambda(w_j,w_{j+1})=c_{j+1}$. This follows from the covering relations of $\abs(B_n)$ given in
the end of Section \ref{relationbn}. Furthermore, we claim that if $z\in[e,w]$ with $z\neq w_{j+1}$
is such that $w_j\to z$, then $\lambda(w_j,w_{j+1})<\lambda(w_j,z)$.
Indeed, in view of the poset isomorphism $\phi: [u,v]  \to [e, u^{-1}v]$ for $u=w_j$ and $v=w$,
this follows from the special case $j=0$ treated earlier. By definition of $\lambda$ and the
construction of $\mathcal{C}_u$, the sequence
\[\left(\lambda(e,w_1),\lambda(w_1, w_2),\dots,\lambda(w_{r-1}, w)\right)\] coincides with $c(w)$.
Moreover, $\mathcal{C}_w$ is the unique maximal chain having this sequence of labels.
This and the fact that the labels of any chain in $[e,w]$ are elements of the set $\{c_1, c_2,\dots,c_r\}$
imply  that $\mathcal{C}_w$ is the unique strictly increasing maximal chain.
By what we have already shown, $\mathcal{C}_w$ lexicographically precedes all other maximal chains of $[e,w]$.
Thus $\mathcal{C}_w$ is lexicographically first and the unique strictly increasing chain in $[e,w]$.
Hence $\lambda$ is an EL-labeling for the interval $[e,w]$ and Proposition \ref{el} is proved.
\end{proof}

\begin{example}
\begin{enumerate}[leftmargin=*, itemsep=5pt]
\item  [(i)] Let $n= 7$ and $w=[1,-7][3]\lleft2,\,-6,\,-5\rright\lleft4\rright\in B_7$.
Then $c(w)=(1,3,5,6,7)$ and
\[\mathcal{C}_w: e\stackrel{1}{\to}[1]\stackrel{3}{\to}[1][3]\stackrel{5}{\to}[1][3]\lleft2,-5\rright\stackrel{6}{\to}[1][3]\lleft2,-6,-5\rright\stackrel{7}{\to} w.\]
\item  [(ii)] Let $n=4$ and $w=[3,-4]\lleft1,2\rright$.
Then $c(w)=(2,3,4)$ and \[\mathcal{C}_w:e\stackrel{2}{\to} \lleft1,2\rright\stackrel{3}{\to} \lleft1,2\rright[3]\stackrel{4}{\to} w.\]
\end{enumerate}
\end{example}

\noindent{\emph{Proof of Theorem \ref{th2}.}\,
It follows from Proposition \ref{el}, since EL-shellability implies shellability.
\qed

\begin{figure}[h]
\begin{center}
\includegraphics[width=4.32in]{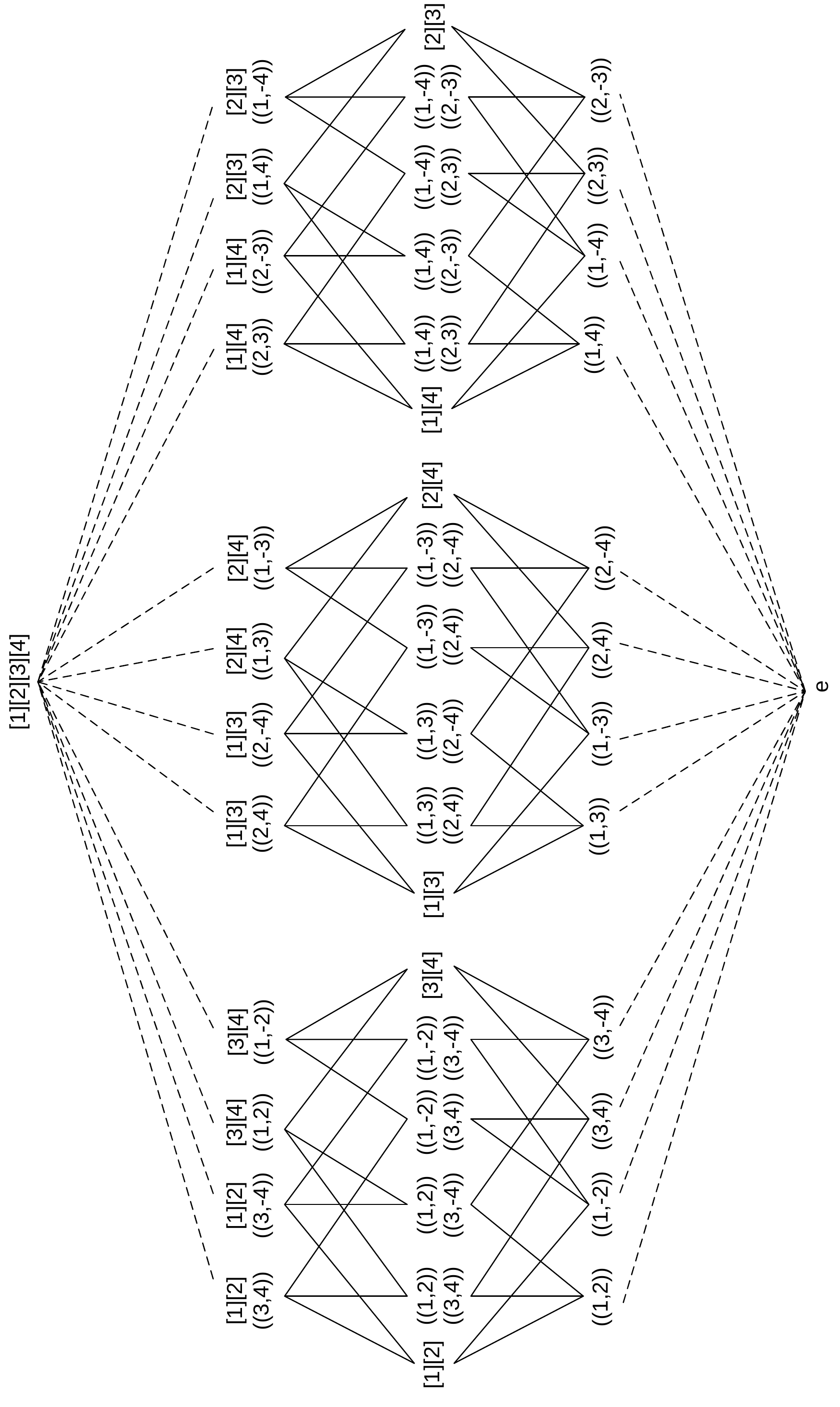}
\end{center}
\os
\label{d4b}
\end{figure}

\begin{remark}
\label{exd4}
Figure \ref{d4b} illustrates the Hasse diagram of the interval $I=\left[e, u\right]$ of
$\abs(D_4)$, where $u= [1][2][3][4]$. Note that the Hasse diagram of the open interval $(e,u)$ is
disconnected and, therefore, $I$ is not Cohen-Macaulay over any field. 
Since $\abs(D_n)$ contains an interval which is isomorphic to $I$ for any $n\geq 4$, it follows that	$\abs(D_n)$
is not Cohen-Macaulay over any field for $n\geq 4$ either (see \cite[Corollary 3.1.9]{mwa}).
\end{remark}



\section{Cohen-Macaulayness}
\label{cmcm}
In this section we prove Theorems \ref{th3} and \ref{th3b}. Our method to show that $\mathcal{J}_n$ is homotopy Cohen-Macaulay is based on 
the following theorem, due to Quillen
\cite[Corollary 9.7]{Q}; see also \cite[Theorem 5.1]{bww}. 
The same method yields a new proof of Theorem \ref{thca1}, which we also include in this section.  

\begin{theorem}
\label{zup}
Let $P$ and $Q$ be graded posets and let $f:P\to Q$ be a surjective rank-preserving poset map. Assume
that for all $q\in Q$ the fiber $f^{-1}(\langle q\rangle)$ is homotopy Cohen-Macaulay. If
$Q$ is homotopy Cohen-Macaulay, then so is $P$.
\end{theorem}
For other poset fiber theorems of this type, see \cite{bww}.

\smallskip

To prove Theorems \ref{thca1} and \ref{th3}, we need the following.
Let $\{\hat{0},\hat{1}\}$ be a two element chain, with
$\hat{0}<\hat{1}$ and $i\in\{1,2,\dots,n\}$. 
We consider the map $\pi_i:P_n\to P_n$ of Section \ref{abs}, where $P_n$ is either $\abs(S_n)$ or $\mathcal{J}_n$. 
We define the map

\[f_i:P_n\to \pi_i(P_n)\times\{\hat{0},\hat{1}\}\]
by letting
\[f_i(w)=\left\{\begin{array}{ll}

(\pi_i(w),\,\hat{0}), & \mbox{if $w(i)=i$}, \\
(\pi_i(w),\,\hat{1}), & \mbox{if $w(i)\neq i$}

\end{array}
\right.\]
for $w\in P_n$. We first check that $f_i$ is a surjective rank-preserving poset map. Indeed, by definition $f_i$
is rank-preserving. Let $u,v\in P_n$ with $u\mik v$. Lemma \ref{proj1} (ii) implies that
$\pi_i(u)\mik \pi_i(v)$. If $u(i)\neq i$, then $v(i)\neq i$ as
well and hence $f_i(u)=(\pi_i(u),\hat{1})\leq(\pi_i(v),\hat{1})=f_i(v)$.
If $u(i)=i$, then $f_i(u)=(\pi_i(u),\hat{0})$ 
and hence $f_i(u)\leq f_i(v)$. Thus $f_i$ is a poset map.
Moreover, if $w\in \pi_i(P_n)$, then $f_i^{-1}\left(\{(w,\hat{0})\}\right)=\{w\}$
and any permutation obtained
from $w$ by inserting the element $i$ in a cycle of $w$ lies in
$f_i^{-1}\left(\{(w,\hat{1}\right)\})$. Thus $f_i^{-1}\left(\{q\}\right)\neq\varnothing$ for every $q\in \pi_i(P_n)\times\{\hat{0},\hat{1}\}$, 
which means that $f_i$ is surjective. 

Given a map $f: P \to Q$, we abbreviate by $f^{-1} (q)$
the inverse image  $f^{-1} (\{q\})$ of a singleton subset $\{q\}$ of $Q$.
For subsets $U$ and $V$ of $S_n$ (respectively, of $B_n$), we write $U\cdot V=\{uv: u\in U, v\in V\}$.

\

The following lemmas will be used in the proof of Theorem \ref{thca1}.

\smallskip

\begin{lemma}
\label{f-1}
For every $q\in S_{n-1}\times \{\hat{0},\hat{1}\}$ we have $f_n^{-1}\left(\langle q\rangle\right)=\langle f_n^{-1}(q)\rangle$.
\end{lemma}


\begin{proof}
The result is trivial for $q=(u,\hat{0})\in S_{n-1}\times\{\hat{0},\hat{1}\}$, so suppose that $q=(u,\hat{1})$.  
Since $f_n$ is a poset map, we have $\langle f_n^{-1}(q)\rangle\subseteq f_n^{-1}\left(\langle q\rangle\right)$. 
For the reverse inclusion consider any element $w\in f_n^{-1}\left(\langle q\rangle\right)$. Then $f_n(w)\leq q$ and hence $\pi_n(w)\mik u$. 
Lemma \ref{antenatel} implies that there exists an element $v\in S_n$ which covers $u$ and satisfies $\pi_n(v)=u$ and $w\mik v$. 
We then have $v\in f_n^{-1}(q)$ and hence $w\in  \langle f_n^{-1}(q)\rangle$. This proves that $f_n^{-1}\left(\langle q\rangle\right)\subseteq\langle f_n^{-1}(q)\rangle$.
\end{proof}

\smallskip

\begin{lemma}
\label{lsn}
For every $u\in S_{n-1}$, the order ideal
\[M(u)=\langle v\in S_n:\pi_n(v)=u\rangle\] of $\abs(S_n)$ is homotopy Cohen-Macaulay of rank $\ell_{\mathcal{T}}(u)+1$.
\end{lemma}

\begin{proof}
Let $u=u_1u_2\cdots u_l$ be written as a product of disjoint cycles in $S_{n-1}$. 
Then \[M(u)=\bigcup\limits_{i=1}\limits^lC(u_i)\cdot\langle u_1\cdots\hat{u}_i\cdots u_l\rangle,\] 
where $u_1\cdots\hat{u}_i\cdots u_l$ denotes the permutation obtained from $u$ by deleting the
cycle $u_i$ and $C(u_i)$ denotes the order ideal of $\abs(S_n)$ generated by the cycles $v$ of $S_n$ which cover $u_i$ and satisfy $\pi_n(v)=u_i$. 
Lemma \ref{l4}, proved in the Appendix, implies that $C(u_i)$ is homotopy Cohen-Macaulay of rank $\ell_{\mathcal{T}}(u_i)+1$ for every $i$. 
Each of the ideals $C(u_i)\cdot\langle u_1\cdots\hat{u}_i\cdots u_l\rangle$ is isomorphic to a direct product of 
homotopy Cohen-Macaulay posets and 
hence it is homotopy Cohen-Macaulay, by Lemma \ref{tomes} (i); their rank is equal to $\ell_{\mathcal{T}}(u)+1$. 
Moreover, the intersection of any two or more of the ideals $C(u_i)\cdot\langle u_1\cdots\hat{u}_i\cdots u_l\rangle$ is equal to 
$\langle u\rangle$, which is homotopy Cohen-Macaulay of rank $\ell_{\mathcal{T}}(u)$. 
Thus the result follows from Lemma \ref{tomes} (ii). 
\end{proof}

\noindent \emph{Proof of Theorem \ref{thca1}.}\ We proceed by induction on $n$. The result is trivial for
$n\leq 2$. Suppose that the poset $\abs(S_{n-1})$ is homotopy Cohen-Macaulay. 
Then so is the direct product
$\abs(S_{n-1})\times\{\hat{0},\hat{1}\}$ by Lemma \ref{tomes} (i). We consider the map
\[f_n:\abs(S_n)\to \abs(S_{n-1})\times\{\hat{0},\hat{1}\}.\]
In view of Theorem \ref{zup} and Lemma \ref{f-1}, it suffices to show that for every
$q\in S_{n-1}\times\{\hat{0},\hat{1}\}$ the order ideal
$\langle f_n^{-1}(q)\rangle$ of $\abs(S_n)$ is homotopy
Cohen-Macaulay.
This is true in case $q=(u,\hat{0})$ for some $u\in S_{n-1}$, since then 
$\langle f_n^{-1}(q)\rangle=\langle u\rangle$ and every interval in $\abs(S_n)$ is shellable.
Suppose that
$q=(u,\hat{1})$. Then $\langle f_n^{-1}(q)\rangle= M(u)$, which is  homotopy Cohen-Macaulay by Lemma \ref{lsn}.
This completes the induction and the proof of the theorem.\qed

\

We now focus on the hyperoctahedral group. 
The proof of Theorem \ref{th3} is based on the following lemmas.

\begin{lemma}
\label{g^{-1}}
For every $q\in \mathcal{J}_{n-1}\times \{\hat{0},\hat{1}\}$ we have $f_n^{-1}\left(\langle q\rangle\right)=\langle f_n^{-1}(q)\rangle$.
\end{lemma}

\begin{proof}
The proof of Lemma \ref{f-1} applies word by word, if one replaces $S_{n-1}$ by the ideal $\mathcal{J}_{n-1}$. We thus omit the details. 
\end{proof}

\begin{lemma}
\label{yohoho}
For every $u\in  \mathcal{J}_{n-1}$ the order ideal
\[M(u)=\langle v\in \mathcal{J}_n:\,\pi_n(v)=u\rangle\]
of $\abs(B_n)$ is homotopy Cohen-Macaulay of rank $\ell_{\mathcal{T}}(u)+1$.
\end{lemma}

\begin{proof}
Let  $u=u_1u_2\cdots u_l\in\mathcal{J}_{n-1}$ be written as a product of disjoint cycles. 
For $i\in\{1,\dots,l\}$, we denote by $C(u_i)$ the order ideal of $\mathcal{J}_n$ generated by 
all cycles $v\in\mathcal{J}_n$ which can be obtained by inserting either $n$ or $-n$ at any place in the cycle $u_i$. 
The ideal $C(u_i)$ is graded of rank $\ell_{\mathcal{T}}(u_i)+1$ 
and  homotopy Cohen-Macaulay, by Lemma \ref{l4'} proved in the Appendix.
Let $u_1\cdots \hat{u}_i\cdots u_l$ denote the permutation obtained from $u$ by removing the cycle $u_i$. 
Suppose first that $u$ has a balanced cycle in its cycle decomposition. Using Remark \ref{anal}, we find that 
\[M(u)=\bigcup\limits_{i=1}\limits^lC(u_i)\cdot\langle u_1\cdots \hat{u}_i\cdots u_l\rangle.\]
Clearly, $M(u)$ is graded of rank $\ell_{\mathcal{T}}(u)+1$. 
Each of the ideals $C(u_i)\cdot\langle u_1\cdots\hat{u}_i\cdots u_l\rangle$ is isomorphic to a direct product of homotopy Cohen-Macaulay posets and 
hence it is homotopy Cohen-Macaulay, by Lemma \ref{tomes} (i). 
Moreover, the intersection of any two or more of these ideals 
is equal to $\langle u\rangle$, which is homotopy Cohen-Macaulay of rank $\ell_{\mathcal{T}}(u)$, by Theorem \ref{th2}. 
Suppose now that $u$ has no balanced cycle in its cycle decomposition. 
Then 
\[M(u)=\bigcup\limits_{i=1}\limits^l C(u_i)\cdot\langle u_1\cdots \hat{u}_i\cdots u_l\rangle\cup\langle u[n]\rangle.\] 
Again,  $M(u)$ is graded of rank $\ell_{\mathcal{T}}(u)+1$,  
each of the ideals $C(u_i)\langle u_1\cdots \hat{u}_i\cdots u_l\rangle$ and  $\langle u[n]\rangle$ is homotopy Cohen-Macaulay and the intersection of any two or more of 
these ideals is equal to $\langle u\rangle$. In either case, the result follows from Lemma \ref{tomes} (ii).
\end{proof}

\noindent\emph{Proof of Theorem \ref{th3}.}\ 
We proceed induction on $n$. The result is trivial for
$n\leq 2$. Suppose that the poset $\mathcal{J}_{n-1}$  is homotopy Cohen-Macaulay. 
Then so is the direct product
$\mathcal{J}_{n-1}\times\{\hat{0},\hat{1}\}$ by Lemma \ref{tomes} (i). We consider the map
\[f_n:\mathcal{J}_n\to\mathcal{J}_{n-1}\times\{\hat{0},\hat{1}\}.\]
In view of Theorem \ref{zup} and Lemma \ref{g^{-1}}, it suffices to show that for every
$q\in\mathcal{J}_{n-1}\times\{\hat{0},\hat{1}\}$ the order ideal
$\langle f_n^{-1}(q)\rangle$ of $\abs(B_n)$ is homotopy
Cohen-Macaulay.
This is true in case $q=(u,\hat{0})$ for some $u\in \mathcal{J}_{n-1}$, since then 
$\langle f_n^{-1}(q)\rangle=\langle u\rangle$ and every interval in $\abs(B_n)$ is shellable by Theorem \ref{th2}.
Suppose that
$q=(u,\hat{1})$. Then $\langle f_n^{-1}(q)\rangle= M(u)$, which is  homotopy Cohen-Macaulay by Lemma \ref{lsn}.
This completes the induction and the proof of the theorem.\qed

\

\noindent\emph{Proof of Theorem \ref{th3b}.}\ 
Let us denote by $\hat{0}$ the minimum element of $\abs(B_n)$.
Let $\hat{\mathcal{J}}_n$ be the poset obtained from $\mathcal{J}_n$ by adding
a maximum element $\hat{1}$ and let $\mu_n$ be the
M\"obius function of $\hat{\mathcal{J}}_n$.
From Proposition 3.8.6 of \cite{St} we have that $\tilde{\chi} (\Delta (\bar{\mathcal{J}}_n)) = \mu_n (\hat{0},
\hat{1})$. Since $\mu_n(\hat{0},\hat{1})=-\sum\limits_{x \in \mathcal{J}_n} \ \mu_n
(\hat{0}, x)$, we have
\begin{equation}
\label{eq1}
\tilde{\chi} (\Delta (\bar{\mathcal{J}}_n)) = -\sum\limits_{x \in \mathcal{J}_n} \ \mu_n
(\hat{0}, x).
\end{equation}

\noindent Suppose that $x\in B_n$ is a cycle. It is known \cite{R} that

\[\mu(\hat{0},x)=\left\{\begin{array}{ll}

(-1)^m{2m-1\choose k}, & \mbox{if $x$ is a balanced $m$-cycle}, \\
(-1)^{m-1}C_{m-1}, & \mbox{if $x$ is a paired $m$-cycle},

\end{array}
\right.\]
where $C_m=\frac{1}{m+1}{2m\choose m}$ is the $m$th Catalan number.
We recall (Remark \ref{anal}) that if $x\in\mathcal{J}_n$
has exactly $k+1$ paired cycles, say $p_1,\dots,p_{k+1}$, and one balanced cycle, say $b$, then 
$[\hat{0},x]\cong [\hat{0},b]\times[\hat{0},p_1]\times\cdots\times[\hat{0},p_k]$
 and hence 
\[\mu_n(\hat{0},x)=\mu_n(\hat{0},b)\,\prod\limits_{i=1}\limits^k\, \mu_n(\hat{0},p_i).\]
It follows that

\begin{equation}
\label{eq2}
\mu_n(\hat{0},x)=(-1)^{\ell_{\mathcal{T}}(b)}{2\ell_{\mathcal{T}}(b)-1\choose \ell_{\mathcal{T}}(b)} \prod\limits_{i=1}\limits^k (-1)^{\ell_{\mathcal{T}}(p_i)}C_{\ell_{\mathcal{T}}(p_i)}.
\end{equation}
From (\ref{eq1}), (\ref{eq2}), \cite[Proposition 5.1.1]{St2} and the exponential formula  \cite[Corollary 5.1.9]{St2}, we conclude that
\begin{equation}
\label{eqa}
1 - \sum_{n \ge 2} \tilde{\chi} (\Delta (\bar{\mathcal{J}}_n)) \frac{t^n}{n!}
= \left( 1+\sum_{n \geq 1}2^{n-1}\al_n\frac{t^n}{n}\right) \exp\left(\sum_{n \geq 1}2^{n-1}\beta_n \frac{t^n}{n}\right),
\end{equation}
where $\al_n=(-1)^n{2n-1\choose n}$ is the M\"obius function of a balanced $n$-cycle and $\beta_n=(-1)^{n-1} C_{n-1}$
is the M\"obius function of a paired $n$-cycle.
Thus it suffices to compute $\exp\left(\sum_{n \geq 1}2^{n-1}\beta_n \frac{t^n}{n}\right)$.
From \cite[ Section 5]{ca} we have that
\[\exp\,\sum_{n \ge 1} \beta_n\frac{t^n}{n}=\frac{\sqrt{1+4t}-1}{2t}\,\exp\left(\sqrt{1+4t}-1\right)\]
and hence, replacing $t$ by $2t$,
\[\exp\left(\sum_{n \geq 1}	 2^{n-1} \beta_n \frac{t^n}{n}\right)=
\left(\frac{\sqrt{1+8t}-1}{4t}\right)^{1/2}\exp\left(\frac{\sqrt{1+8t}-1}{2}\right).\]
The right-hand side of (\ref{eqa}) can now be written as
\[
1-\left(\frac{\sqrt{1+8t}-1}{4t}\right)^{1/2}\exp\left(\frac{\sqrt{1+8t}-1}{2}\right)\left(1+\sum_{n \geq 1}2^{n-1}\al_n\frac{t^n}{n}\right).\]

\noindent The result follows by switching $t$ to $-t$.
\qed

\

\begin{remark}
Theorem \ref{th3} can also be proved using the notion of strong constructibility,
introduced in \cite{ca}. The details will appear in \cite{myr}.
\end{remark}


\section{Intervals with the lattice property}
\label{latlat}
Let $W$ be a finite Coxeter group and $c\in W$ be a Coxeter element. It is known \cite{Be,brwtt,brwtt0} that the interval $[e,c]$ 
in $\abs(W)$ is a lattice. 
In this section we characterize the intervals in $\abs(B_n)$ and $\abs(D_n)$ which are lattices (Theorems \ref{th5} and \ref{th6}, respectively). 
As we explain in the sequel, some partial results in this direction were obtained in \cite{Be, brwtt, brwtt0, N, R}.

To each $w\in B_n$ we associate the integer partition $\mu(w)$ whose parts are the absolute lengths of all balanced cycles of $w$, 
arranged in decreasing order. 
  For example, if $n=8$ and $w=[1,-5][2,7][6]\lleft 3,4\rright$, then $\mu(w)=(2,2,1)$. 
It follows from the results of \cite[Section 6]{N} that the interval $[e,w]$ in $\abs(B_n)$ is a lattice if $\mu(w)=(n-1,1)$ 
and that $[e, w]$ is not a lattice if $\mu(w)=(2,2)$. 
Recall that a \emph{hook partition} is an integer partition of the form $\mu=(k,1,\dots,1)$, also written as $\mu=(k,1^r)$, where $r$ is one less than the total number of parts of $\mu$. 
Our main results in this section are the following.

\begin{theorem}
\label{th5}
For $w\in B_n$, the interval $[e,w]$ in $\abs(B_n)$ is a lattice if and only if $\mu(w)$ is a hook partition. 
\end{theorem}
\begin{theorem}
\label{th6}
For $w\in D_n$, the interval $[e,w]$ in $\abs(D_n)$ is a lattice if and only if  $\mu(w)=(k,1)$ for some $k\leq n-1$, or $\mu(w)=(1,1,1,1)$. 
\end{theorem}

We note that in view of Lemma \ref{lemarm}, Theorems \ref{th5} and \ref{th6} characterize all closed intervals in $\abs(B_n)$ and $\abs(D_n)$ which are lattices. 
The following proposition provides one half of the first characterization.

\begin{propo}
\label{charlatt}
Let $w\in B_n$. If $\mu(w)$ is a hook partition, 
then the interval $[e,w]$ in $\abs(B_n)$ is a lattice.
\end{propo}

\begin{proof}
Let us write $w=bp$, where $b$ (respectively,  $p$) is the product of all balanced  (respectively, paired) cycles of $w$.  
We recall then that $[e,w]\cong[e,b]\times[e,p]$ (see Remark \ref{anal}). Since $[e,p]$ is isomorphic to a direct product of noncrossing partition lattices, 
the interval $[e,w]$ is a lattice if and only if $[e,b]$ is a lattice. 
Thus we may assume that $w$ is a product of disjoint balanced cycles. 
Since $\mu(w)$ is a hook partition, we may further assume that $w=[1,2,\dots,k][k+1]\cdots[k+r]$  with $k+r\leq n$. 
We will show that  $L(k,r):=[e,w]$ is a lattice by 
induction on $k+r$. The result is trivial for $k+r=2$. 
Suppose that  $k+r\geq 3$ and that the poset $L(k,r)$ is a lattice whenever $k+r< \kappa+\rho\leq n$.
We will show that $L(\kappa, \rho)$ is a lattice as well. 
For $\rho\leq 1$, this follows from  \cite[Proposition 2]{R} and the result of \cite{N} mentioned earlier.  
Thus we may assume that $\rho\geq 2$. Let $u,v\in L(\kappa,\rho)$. 
By \cite[Proposition 3.3.1]{St}, it suffices to show that $[e, u]\cap [e, v]=[e, z]$ for some $z\in L(\kappa, \rho)$.

Suppose first that $u(i)=i$ for some $i\in\{1,2,\dots, \kappa + \rho\}$ 
and let $v'$ be the signed permutation obtained by deleting the element $i$ from 
the cycle decomposition of $v$. 
We may assume that $u,v'\in L(\kappa_1,\rho_1)$,
where either $\kappa_1= \kappa-1$ and $\rho_1= \rho$, or $\kappa_1= \kappa $ and $\rho_1= \rho-1$.
We observe that $[e, u]\cap[e, v]=[e, u]\cap[e,v']$.
Since  $L(\kappa_1,\rho_1)$ is a lattice by induction, there exists an element
$z\in L(\kappa_1,\rho_1)$ such that $[e, u]\cap[e, v']=[e, z]$.
We argue in a similar way if $v(i)=i$ for some	$i\in\{1,2,\dots, \kappa + \rho\}$.

Suppose that $u(i)\neq i$ and $v(i)\neq i$ for every $i\in\{1,2,\dots, \kappa + \rho\}$. 
Since $\rho\geq 2$, each of $u,v$ has at least  one reflection in its cycle decomposition. 
Without loss of generality, we may assume that
no cycle of $u$ is comparable to a cycle of $v$ in $\abs(B_n)$ (otherwise the result follows by induction).
Then at least one of the following holds:

\begin{itemize}[leftmargin=*, itemsep=5pt]
\item The reflection $[i]$ is a cycle of $u$ or $v$ for some $i\in\{\kappa +1, \kappa +2,\dots, \kappa + \rho\}$.
\item There exist $i,j\in\{\kappa +1, \kappa +2,\dots, \kappa + \rho\}$ with $i< j$, such that either
$\lleft i,j\rright$ or $\lleft i,-j\rright$ is a cycle of $u$ and  $i$ and $j$ belong to distinct cycles of $v$, or conversely.
\item There exist $i,j\in\{\kappa +1, \kappa +2,\dots, \kappa + \rho\}$ with $i< j$, such that $\lleft i,j\rright$ is a cycle of $u$
and $\lleft i,-j\rright$ is a cycle of $v$, or conversely.
\end{itemize}
In any of the previous cases, let $u'$ and $v'$ be the permutations obtained from
$u$ and $v$, respectively, by deleting the element $i$ from their cycle decomposition.
We may assume once again that $u',v'\in L(\kappa_1,\rho_1)$,
where either $\kappa_1= \kappa-1$ and $\rho_1=\rho$, or $\kappa_1= \kappa $ and $\rho_1=\rho-1$.
As before, $[e, u]\cap[e, v]=[e,u']\cap[e, v']$.
By the induction hypothesis, $L(\kappa_1,\rho_1)$ is a lattice and hence
$[e,u']\cap[e,v']=[e,z]$
for some $z\in L(\kappa_1,\rho_1)$.
This implies that $L(\kappa,\rho)$ is a lattice and completes the induction. 
\end{proof}


\noindent\emph{Proof of Theorem \ref{th5}}. 
If $\mu(w)$ is a hook partition, then the result follows from Proposition \ref{charlatt}.  
To prove the converse, assume that $w$ has at least two balanced cycles, say $w_1$ and $w_2$, with 
$\ell_{\mathcal{T}}(w_1),\ell_{\mathcal{T}}(w_2)\geq 2$.
Then there exist $i,j,l,m\in \{\pm1,\pm2,\dots,\pm n\}$ with $|i|,|j|,|l|,|m|$ pairwise distinct, such that
$[i,j]\mik w_1$ and $[l,m]\mik w_2$.
However, in \cite[Section 5]{N0} it was shown that the poset $\left[e,[i,j][l,m]\right]$ is not a lattice.
It follows that neither $[e,w]$ is a lattice. This completes the proof.
\qed

\

In the sequel we denote by $L(k,r)$ the lattice $[e,w]\subset\abs(B_n)$, where 
$w=[1,2,\dots,k][k+1]\cdots [k+r]\in B_n$. 
Clearly, $L(k,r)$ is isomorphic to any interval of the form $[e,u]$, where $u\in B_n$ has no paired cycles and satisfies $\mu(u)=(k,1^r)$.

\

\noindent\emph{Proof of Theorem \ref{th6}}. 
The argument in the proof of Theorem \ref{th5}  shows that
the interval $[e,w]$ is not a lattice unless $\mu(w)$ is a hook partition. 
Moreover, it is known \cite{Be,brwtt} that $[e, w]$
is a lattice if $\mu(w)=(k,1)$ for some $k \ge 1$. 
Suppose that $\mu(w)=(k,1^r)$, where $r>1$ and $r+k\leq n$. 
If $k \ge 2$, then there exist distinct elements 
of $[e,w]$ of the form $u=[a_1,a_2][a_3]$ and $v=[a_1,a_2][a_4]$. 
The intersection $[e, u]\cap [e, v]\subset \abs(D_n)$ has two maximal
elements, namely the paired reflections $\lleft a_1,a_2\rright$ and $\lleft a_1,-a_2\rright$. 
This implies that $u$ and $v$ do not have a meet and therefore
the interval $[e, w]$ is not a lattice. 
Suppose that $k=1$. 
Without
loss of generality, we may assume that $[1][2]\cdots[r+1]\mik w$. Suppose that $r+1\geq 5$.
We consider the elements $u=[1][2][3][4]$ and $v=[1][2][3][5]$ of $[e,w]$
and note that the intersection $[e, u]\cap[e, v]$ has three maximal
elements, namely $[1][2],[1][3]$ and $[2][3]$. This implies that the interval $[e,w]$ is not a lattice. 
Finally, if $r+1=4$, then $\mu(w)=(1,1,1,1)$ and $[e,w]=[e,[1][2][3][4]]\times[e,p]$, where $p$ is a product 
of disjoint paired cycles which fixes each $i\in\{1,2,3,4\}$. 
Figure \ref{d4b} shows that the interval
$[e,[1][2][3][4]]$ is a lattice and hence, so is $[e, w]$. This
completes the proof.
\qed


\section{The lattice $\mathcal{L}_n$}
\label{spespe}
The poset $L(n,0)$ is the interval $[e,c]$ of $\abs(B_n)$, where $c$ is the Coxeter element $[1,2,\dots,n]$ of $B_n$.
This poset is isomorphic to the lattice $NC^B(n)$ of noncrossing partitions of type $B$.
Reiner \cite{R} computed its basic enumerative invariants listed below:

\begin{itemize}[leftmargin=*, itemsep=5pt]

\item The cardinality of $NC^B(n)$ is equal to ${2n\choose n}$.
\item The number of elements of rank $k$ is equal to ${n\choose k}^2$.
\item The zeta polynomial satisfies $Z(NC^B(n),m)={mn\choose n}$.
\item The number of maximal chains is equal to $n^n$.
\item The M\"obius function satisfies $\mu_n(\hat{0},\hat{1})=(-1)^n {2n-1\choose n}$.
\end{itemize}

\

In this section we focus on the enumerative properties of another interesting special case of
$L(k,r)$, namely the lattice $\mathcal{L}_n:=L(0,n)$. 
First we describe this poset explicitly. 
Each element of $\mathcal{L}_n$ can be obtained from $[1][2]\cdots[n]$ by applying repeatedly the following steps:

\begin{itemize}
\item delete some $[i]$,
\item replace a product $[i][j]$ with $\lleft i,j\rright$ or $\lleft i,-j\rright$.
\end{itemize}
Thus $w\in \mathcal{L}_n$ if and only if every nontrivial cycle of $w$ is a reflection. 
In that case there is a poset isomorphism $[e,w]\cong\mathcal{L}_k\times \mathcal{B}_{l}$, 
where $k$ and $l$ are the numbers of balanced and paired cycles of $w$, respectively and $\mathcal{B}_l$ denotes the lattice of subsets of the set 
$\{1,2,\dots,l\}$, ordered by inclusion. 
It is worth pointing out that $\mathcal{L}_n$ coincides with the subposet
of $\abs(B_n)$ induced on the set of involutions. 
Figure \ref{el3} illustrates the Hasse diagram of
$\mathcal{L}_3$.

\begin{figure}[h]
\begin{center}
\includegraphics[width=5.2in]{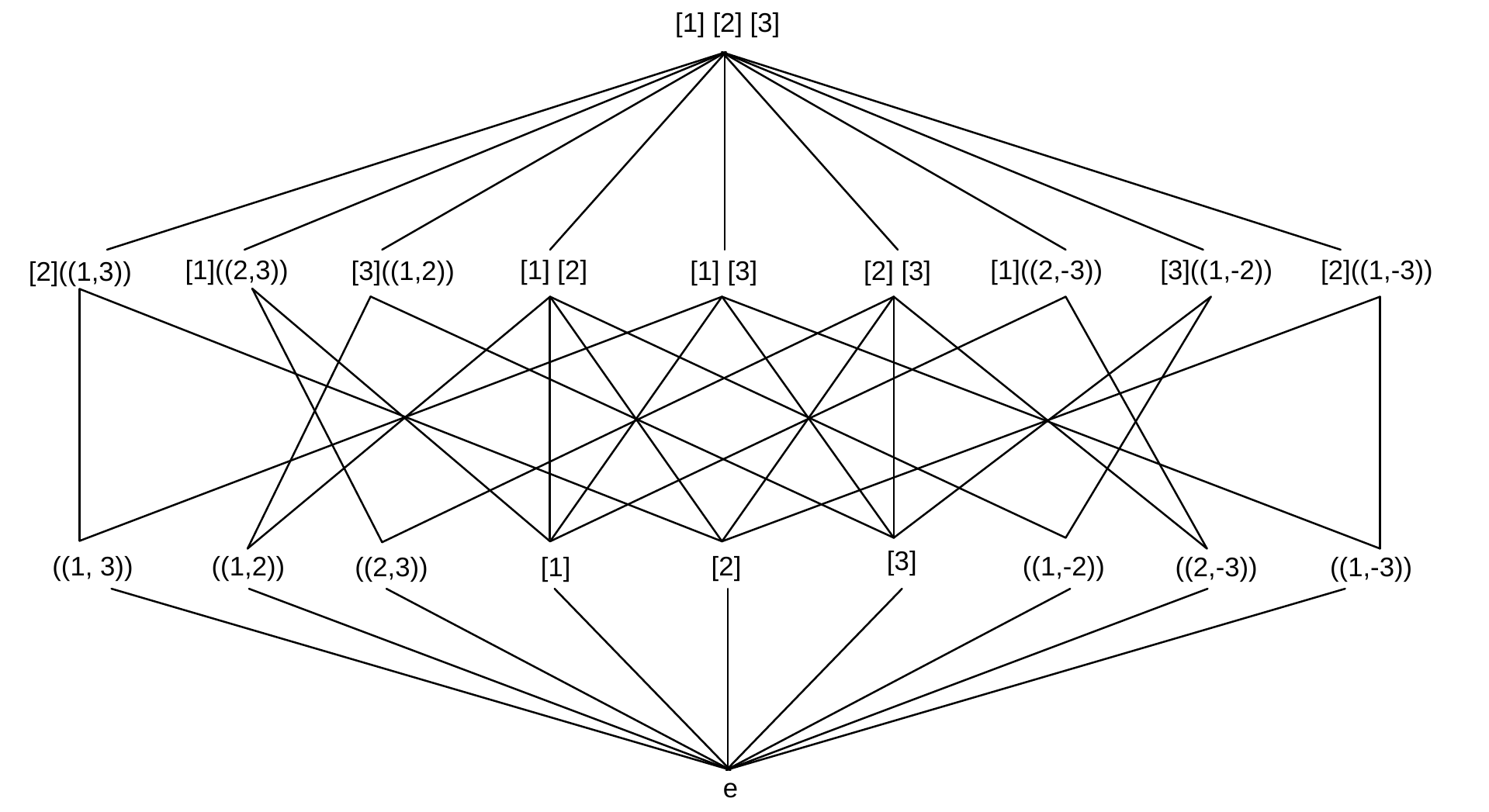}
\end{center}
\os
\label{el3}
\end{figure}

In Proposition \ref{k=0} we give the analogue of the previous list for the lattice $\mathcal{L}_n$.
We recall that the zeta polynomial $Z(P,m)$ of a finite poset $P$ counts the number of multichains	 $x_1\leq x_2\leq\cdots\leq x_{m-1}$ of $P$.
It is known	 (see  \cite{Ed}, \cite[Proposition 3.11.1]{St}) that $Z(P,m)$ is a polynomial function of $m$ of degree $n$, where $n$ is the length of $P$ and that $Z(P,2)=\#P$.
Moreover, the leading coefficient of $Z(P,m)$ is equal to the number of maximal chains divided by $n!$ and if $P$ is bounded, then  $Z(P,-1)=\mu(\hat{0},\hat{1})$.

\begin{propo}
\label{k=0}
\emph{For the lattice $\mathcal{L}_n$ the following hold:}
\begin{enumerate}[leftmargin=*, itemsep=5pt]

\item  [\emph{(i)}]
The number of elements of $\mathcal{L}_n$ is equal to
\[\sum\limits_{k=0}\limits^{\left\lfloor \nicefrac{n}{2}\right\rfloor} {n\choose 2k}\,2^{n-k}(2k-1)!!,\]
where $(2m-1)!!=1\cdot3\,\cdots\, (2m-1)$ for positive integers $m$.

\item  [\emph{(ii)}] The number of elements of $\mathcal{L}_n$ of rank $r$ is equal to

\end{enumerate}
\[\dsp\sum\limits_{k=0}\limits^{\min\{r,n-r\}}\frac{n!}{k!(r-k)!(n-r-k)!}.\]

\begin{enumerate}[leftmargin=*, itemsep=5pt]
\item  [\emph{(iii)}]
The zeta polynomial $Z_n$ of $\mathcal{L}_n$ is given by the formula
\[Z_n(m)=\dsp\sum\limits_{k=0}\limits^{\left\lfloor \nicefrac{n}{2}\right\rfloor} {n\choose 2k}\,m^{n-k}(m-1)^k(2k-1)!!.\]
\item  [\emph{(iv)}]
The number of maximal chains of $\mathcal{L}_n$ is equal to\
\[n!\sum\limits_{k=0}\limits^{\left\lfloor \nicefrac{n}{2}\right\rfloor} {n\choose 2k}(2k-1)!!.\]
\item  [\emph{(v)}]
For the M\"obius function $\mu_n$ of $\mathcal{L}_n$ we have
\[\dsp\mu_n(\hat{0},\hat{1})=(-1)^n\sum\limits_{k=0}\limits^{\left\lfloor \nicefrac{n}{2}\right\rfloor} {n\choose 2k}\,2^k(2k-1)!!,\]
where $\hat{0}$ and $\hat{1}$ denotes the minimum and the maximum element of $\mathcal{L}_n$, respectively.
\end{enumerate}
\end{propo}

\begin{proof}
Suppose that $x$ has $k$ paired reflections.
These can be chosen in $2^k{n\choose 2k}(2k-1)!!$ ways.
On the other hand, the balanced reflections of $w$ can be chosen in $2^{n-2k}$ ways. Therefore the cardinality of $\mathcal{L}_n$ is equal to
\[\sum\limits_{k=0}\limits^{\left\lfloor \nicefrac{n}{2}\right\rfloor} {n\choose 2k}\,2^{n-k}(2k-1)!!.\]
The same argument shows that the number of elements of $\mathcal{L}_n$ of rank $r$, where $r\leq \left\lfloor \nicefrac{n}{2}\right\rfloor$ is equal to

\begin{eqnarray*}
\sum\limits_{k=0}\limits^r2^k{n\choose 2k}(2k-1)!!\,{n-2k\choose r-k}&=&\sum\limits_{k=0}\limits^r2^k{n\choose 2k}\frac{(2k)!}{2^k\,k!}\,{n-2k\choose r-k}\\
&=&\sum\limits_{k=0}\limits^r\frac{n!}{k!(r-k)!(n-r-k)!}.
\end{eqnarray*}
Since $\mathcal{L}_n$ is self dual, the number of elements in $\mathcal{L}_n$ of rank
$r$ is equal to the number of those that have rank $n-r$.
The number of multichains in $\mathcal{L}_n$ in which	 $k$ distinct paired reflections appear, is equal to ${n\choose 2k}(2k-1)!!(m(m-1))^k m^{n-2k}$. Therefore, the zeta polynomial of $\mathcal{L}_n$ is given by
\[Z_n(m)=\sum\limits_{k=0}\limits^{\left\lfloor \nicefrac{n}{2}\right\rfloor} {n\choose 2k}(2k-1)!!\,m^{n-k}(m-1)^k.\]
Finally, computing the coefficient of $m^n$ in this expression for $Z_n(m)$ and multiplying by $n!$ we conclude that
the number of maximal chains of $\mathcal{L}_n$ is equal to
\[n!\sum\limits_{k=0}\limits^{\left\lfloor \nicefrac{n}{2}\right\rfloor} {n\choose 2k}(2k-1)!!\]
and setting $m=-1$ we get
\[\mu_n(\hat{0},\hat{1})=Z_n(-1)=(-1)^n\sum\limits_{k=0}\limits^{\left\lfloor \nicefrac{n}{2}\right\rfloor} {n\choose 2k}(2k-1)!!\,2^k.\]
\end{proof}

\begin{remark}

\noindent By Proposition \ref{el}, the lattice $\mathcal{L}_n$ is EL-shellable.
We describe two more EL-labelings for $\mathcal{L}_n$.
\begin{enumerate}[leftmargin=*, itemsep=5pt]
\item [(i)] Let $\Lambda=\{[i]: i=1,2,\dots,n\}\cup\{\lleft i,j\rright:i,j=1,2,\dots,n,\,i<j\}$.
We linearly order the elements of $\Lambda$ in the following way. We first order the balanced reflections
so that $[i]<_{\Lambda}[j]$ if and only if $i<j$. Then we order the paired reflections lexicographically.
Finally, we define $[n]<_{\Lambda}\lleft1,2\rright$.
The map $\lambda_1:C(B_n)\to \Lambda$ defined as:
\[\lambda_1(a,b)=\left\{\begin{array}{ll}

[i] & \mbox{if $a^{-1}b=[i]$}, \\
\lleft i,j\rright & \mbox{if $a^{-1}b=\lleft i,j\rright$ or $\lleft i,-j\rright$}

\end{array}
\right.\]
is an EL-labeling for $\mathcal{L}_n$.
\end{enumerate}
\begin{enumerate}[leftmargin=*, itemsep=5pt]
\item [(ii)] Let $\mathcal{T}$ be the set of reflections of $B_n$.
We define a total order $<_{\mathcal{T}}$ on $\mathcal{T}$
which extends the order $<_{\Lambda}$, by ordering the reflections
$\lleft i,-j\rright$, for $1\leq i<j\leq n$, lexicographically and letting
$\lleft n-1,n\rright<_{\mathcal{T}}\lleft1,-2\rright$.
For example, if $n=3$ we have the order
$[1]_{\mathcal{T}}<_{\mathcal{T}}[2]<_{\mathcal{T}}[3]<_{\mathcal{T}}\lleft1,2\rright<_{\mathcal{T}}\lleft1,3\rright<_{\mathcal{T}}\lleft2,3\rright<_{\mathcal{T}}\lleft1,-2\rright<_{\mathcal{T}}\lleft1,-3\rright<_{\mathcal{T}}\lleft2,-3\rright$.
Let $t_i$ be the $i$-th reflection in the order above. We define a map $\lambda_2:C(B_n)\to \{1,2,\dots,n^2\}$ as:
\end{enumerate}
\[\lambda_2(a,b)=\min\limits_{1\leq i\leq n^2}\{i:\,t_i\vee a=b\}.\]
\begin{enumerate}[leftmargin=*, itemsep=5pt]
\item []The map $\lambda_2$ is an EL-labeling for $\mathcal{L}_n$.
\end{enumerate}
See Figure \ref{elp(0,2)n} for an example of these two EL-labelings when $n=2$.
\end{remark}

\begin{figure}[h]
\begin{center}
\includegraphics[width=4.5in]{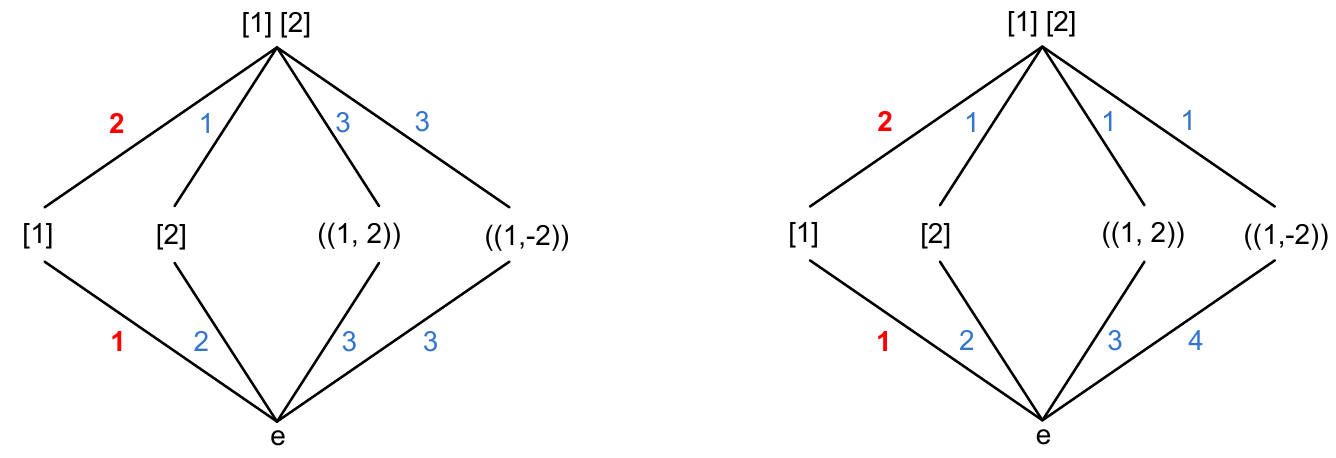}
\end{center}
\os
\label{elp(0,2)n}
\end{figure}


\section{Enumerative combinatorics of $L(k,r)$}
\label{telostelos}
In this section we compute the cardinality, zeta polynomial and M\"obius function  of the lattice $L(k,r)$,
where $k,r$ are nonnegative integers with $k+r=n$.
The case $k=n-1$ was treated by Goulden, Nica and Oancea in their work \cite{N} on the posets of annular noncrossing partitions; 
see also \cite{KM,N0} for related work. We will use their results, as well as the formulas for cardinality and zeta polynomial
for $NC^B(n)$ and Proposition \ref{k=0},
to find the corresponding formulas for $L(k,r)$.


\begin{propo}
\label{k,r}
Let $\al_r=|\mathcal{L}_r|,\,\beta_r(m)=Z(\mathcal{L}_r,m)$
and $\mu_r=\mu_r(\mathcal{L}_r)$, where $\al_r=\beta_r(m)=\mu_r=1$ for $r=0,1$.
For fixed nonnegative integers $k,r$ such that $k+r=n$, the cardinality, zeta polynomial and M\"obius function of $L(k,r)$
are given by:
\begin{itemize}[leftmargin=*, itemsep=5pt]
\item
$\#L(k,r)=\dsp {2k\choose k}\left(\frac{2\,r\,k}{k+1}\,\al_{r-1}+a_r\right)$.

\item
$Z(L(k,r),m)=\dsp{mk\choose k}\left( \frac{2\,r\,k}{k+1}(m-1)\,\beta_{r-1}(m)+\beta_r(m)\right)$.

\item $\mu(L(k,r))=\dsp (-1)^n{2k-1\choose k}\left(\frac{4\,r\,k}{k+1}\,|\mu_{r-1}| + |\mu_r|\right)$.
\end{itemize}
\end{propo}

\begin{proof}


We denote by $A$ the subset of $L(k,r)$ which consists of the elements $x$ with the following property:
every cycle of $x$ that contains at least one of $\pm1,\pm2,\dots,\pm k$ is less than or equal to
the element $[1,2,\dots,k]$ in $\abs(B_n)$.
Let $x=x_1x_2\cdots x_{\nu}\in A$, written as a product of disjoint cycles.
Without loss of generality, we may assume that there is a
$\,t\in\{0,1,\dots,\nu\}$ such that $x_1x_2\cdots x_t\mik [1,2,\dots,k]$ and
$x_{t+1}x_{t+2}\cdots x_{\nu}\mik[k+1][k+2]\cdots[k+r]$.
Observe that if $t=0$ then $x\mik[k+1][k+2]\cdots[k+r]$ in $\abs(B_n)$,
while if $t=\nu$ then $x\mik[1,2,\dots, k]$.
Clearly, there exists a poset isomorphism
\begin{alignat*}{3}
f: A & \to NC^B(k) \ \ && \times && \ \ \langle[k+1]\cdots[k+r]\rangle \\
   x & \mapsto (x_1\cdots x_t \ \ && \	, && \ \ \,x_{t+1}\cdots x_{\nu}),
\end{alignat*}


\noindent so that
\begin{equation}
\label{A}
A\cong NC^B(k)\times \mathcal{L}_r.
\end{equation}

Let $C=L(k,r)\sm A$ and $x=x_1x_2\cdots x_{\nu}\in C$, written as a product of disjoint cycles.
Then there is a exactly one paired cycle $x_1$ of $x$ and one reflection $\lleft i, l\rright$ with
$i\in\{\pm 1,\pm2,\dots, \pm k\},\ l\in\{k+1,k+2,\dots,k+r\}$, such that $\lleft i, l\rright\mik x_1$.
For every $l\in \{k+1,k+2,\dots,k+r\}$ denote by $C_l$ the set of permutations
$x\in L(k,r)$ which have a cycle, say $x_1$, such that	 $\lleft i, l\rright\mik x_1$ for some
$i\in \{\pm 1,\pm 2,\dots,\pm k\}$.
It follows that $C_l\cap C_{l'}=\varnothing$ for $l\neq l'$.
Clearly, $C_l\cong C_{l'}$ for $l\neq l'$ and  $C=\bigcup_{l=k+1}^{k+r}C_l$.

Summarizing, for every $x\in C$ there exists an ordering
$x_1,x_2,\dots,x_{\nu}$ of the cycles of $x$ and a unique index $t\in\{1,2,\dots,\nu\}$
such that $x_1x_2\cdots x_t\mik [1,2,\dots,k][l]$ and
$x_{t+1}x_{t+2}\cdots x_{\nu}\mik[k+1][k+2]\cdots[l-1][l+1]\cdots[k+r]$.
Let	 \[E_l=\{x\in C: x\mik[1,2,\dots,k][l]\}.\] We remark that no permutation of $E_l$ has a balanced cycle in its cycle decomposition.
Clearly, there exists a poset isomorphism
\begin{alignat*}{3}
g_l: C_l & \to\ \ \ \ \, E_l  \ \ && \times && \ \ \langle[k+1]\cdots[l-1][l+1]\cdots[k+r]\rangle \\
   x & \mapsto (x_1\cdots x_t \ \ && \	, && \ \ x_{t+1}\cdots x_{\nu})
\end{alignat*}


\noindent so that
\begin{equation}
\label{C_l}
C_l\cong E_l\times \mathcal{L}_{r-1}
\end{equation}
for every $l\in\{k+1,k+2,\dots,k+r\}$. Using (\ref{A}) and (\ref{C_l}), we proceed to the proof of Proposition \ref{k,r}
as follows. 
From our previous discussion we have $L(k,r)=\#A+r\,(\#C_{k+1})$.
From (\ref{A}) we have	\[\#A={2k\choose k}\al_r\]
and (\ref{C_l}) implies that $\#C_{k+1}=(\#E_{k+1})\,(\#\mathcal{L}_{r-1})=(\#E_{k+1})\, \al_{r-1}$.
Since $E_{k+1}$ consists of the permutations in $\langle[1,2,\dots,k][k+1]\rangle\cap C$,
it follows from \cite[Section 5]{N} that $\#E_{k+1}=2{2k\choose k-1}$.
Therefore,
\[\dsp \#L(k,r)= 2\,r\,{2k\choose k-1}\al_{r-1}+{2k\choose k}\al_r= {2k\choose k}\left(\frac{2r k}{k+1} \al_{r-1}+a_r\right).\]

\

Recall that the zeta polynomial $Z(L(k,r),m)$ counts the number of multichains
$\pi_1\mik \pi_2\mik\cdots\mik\pi_{m-1}$ in $L(k,r)$.
We distinguish two cases. If $\pi_{m-1}\in C$, then $\pi_{m-1}\in C_l$ for some $l\in \{k+1,\dots,k+r\}$.
Isomorphism (\ref{C_l}) then implies that there are $Z(E_l,m)\, Z(\mathcal{L}_{r-1},m)$ such multichains.
From \cite[Section 5]{N} we have $Z(E_l,m)=2{mk\choose k+1}$, therefore
$Z(E_l,m)\, Z(\mathcal{L}_{r-1},m)=2{mk\choose k+1} \beta_{r-1}$.
Since there are $r$ choices for the set $C_l$, we conclude that	 the number of multichains
$\pi_1\mik \pi_2\mik\cdots\mik\pi_{m-1}$ in $L(k,r)$ for which $\pi_{m-1}\in C$ is equal to
\begin{equation}
\label{z1}
2\, r\,{mk\choose k+1} \beta_{r-1}(m).
\end{equation}

\noindent If $\pi_{m-1}\in A$, then $\pi_{m-1}\in NC^B(k)\times \mathcal{L}_r$ and therefore
number of such multichains is equal to

\begin{equation}
\label{z2}
{mk\choose k} \beta_r(m).
\end{equation}

\noindent The proposed expression for the zeta polynomial of $L(k,r)$ follows by summing the expressions (\ref{z1}) and (\ref{z2}) 
and straightforward calculation.

The expression for the M\"obius function follows once again from that of the zeta polynomial by
setting $m=-1$.
\end{proof}


\section{Appendix}
\label{app}

In this section we prove the following lemmas.
\begin{lemma}
\label{l4}
The order ideal of $\abs(S_n)$ generated by all cycels $u\in S_n$ for which $\pi_n(u)=(1\,2\,\cdots\,n-1)$ is homotopy Cohen-Macaulay of rank $n-1$.
\end{lemma}

\begin{lemma}
\label{l4'}
\begin{enumerate}[leftmargin=*, itemsep=5pt]
\item[\emph{(i)}] The order ideal of $\abs(B_n)$ generated by all cycles $u\in B_n$ for which $\pi_n(u)=\lleft1,2,\dots,n-1\rright$ is homotopy Cohen-Macaulay of rank $n-1$.
\item[\emph{(ii)}] The order ideal of $\abs(B_n)$ generated by all cycles $u$ of $B_n$ for which $\pi_n(u)=[1,2,\dots,n-1]$ is homotopy Cohen-Macaulay of rank $n$.
\end{enumerate}
\end{lemma}

\smallskip

\subsection{Proof of Lemma \ref{l4}}

\noindent We will show that the order ideal considered in Lemma \ref{l4} is in fact strongly constructible. 
The following remark will be used in the proof. 
\begin{remark}
\label{L0}
Let $u_1,u_2,\dots, u_m\in S_n$ be elements of absolute length $k$ and let $v\in S_n$ be a cycle of absolute length $r$ 
which is disjoint from $u_i$ for each $i\in\{1,2,\dots, m\}$. 
Suppose that the union $\bigcup_{i=1}^m[e,u_i]$ is strongly constructible of rank $k$. 
Then 
\[\bigcup\limits_{i=1}\limits^m\,[e,vu_i]\,\cong\,\bigcup\limits_{i=1}\limits^m\left([e,v]\times[e,u_i]\right)\,=\,[e,v]\times \bigcup\limits_{i=1}\limits^m\,[e,u_i],\] 
is strongly constructible of rank $k+r$, by Lemma \ref{strcon} (i). 
\end{remark}

\begin{lemma}
\label{l1}
For $i\in\{1,2,\dots,n-1\}$, consider the element \[u_i=(1\,i+1\,\cdots\,n-1)(2\,3\,\cdots\,i\, n)\in S_n.\] 
The union  
$\bigcup_{i=1}^m[e,u_i]$ is strongly constructible of rank $n-2$ for all $1\leq m\leq n-1$.
\end{lemma}

\begin{proof}
We denote by $I(n,m)$ the union in the statement of the lemma and proceed by induction on $n$ and $m$, in this order.
We may assume that $n\geq 3$ and $m\geq 2$, since otherwise the result is trivial. 
Suppose that the result holds for positive integers smaller than $n$. 
We will show that it holds for $n$ as well. 
By induction on $m$, it suffices to show that $[e,u_m]\cap I(n,m-1)$ is strongly constructible of rank $n-3$. 
Indeed, we have $[e,u_m]\cap I(n,m-1)=\bigcup_{i=1}^{m-1}[e,u_m]\cap[e,u_i]$
and \[[e,u_m]\cap[e,u_i]=[e,(1\,m+1\,m+2\cdots\,n-1)(2\,3\,\cdots\,i\,n)(i+1\,i+2\,\cdots\,m)].\]
Since the cycle $(1\,m+1\,m+2\cdots\,n-1)$ is present in the disjoint cycle decomposition of each maximal element of 
$[e,u_m]\cap I(n,m-1)$, the desired statements follows easily from Remark \ref{L0} by induction on $n$.
\end{proof}

\begin{example}
If $n=6$ and $m=3$, then $I(n,m)$ is the order ideal of $\abs(S_n)$ generated by the elements $u_1=(1\,2\,3\,4\,5)(6),\,u_2=(1\,3\,4\,5)(2\,6)$ 
and $u_3=(1\,4\,5)(2\,3\,6)$. 
The intersection \[[e,u_3]\cap\left([e,u_1]\cup[e,u_2]\right)\,=\,[e,(1\,4\,5)(2\,3)(6)]\cup[e,(1\,4\,5)(3)(2\,6)]\]
is strongly constructible of rank $3$ and $I(n,m)$ is strongly constructible of rank $4$.
\end{example}

\begin{lemma}
\label{l2}
For $i\in\{1,2,\dots,n-2\}$, consider the element  
\[v_i=(1\,n\,i+2\,\cdots\,n-1)(2\,3\,\cdots\,i+1)\in S_n.\] 
The union  
$\bigcup_{i=1}^m[e,v_i]$ is strongly constructible of rank $n-2$ for all $1\leq m\leq n-2$.
\end{lemma}

\begin{proof}
The proof is similar to that of Lemma \ref{l1} and is omitted.
\end{proof}

\begin{lemma}
\label{l3}
Let $u_1,u_2,\dots,u_{n-1}\in S_n$ and $v_1,v_2,\dots,v_{n-2}\in S_n$ be defined as in Lemmas \ref{l1} and \ref{l2}, respectively. 
If $I_n=\bigcup_{i=1}^{n-1}[e,u_i]$ and $I'_n=\bigcup_{i=1}^{n-2}[e,v_i]$,
then $I_n\cap I'_n$ is strongly constructible of rank $n-3$.
\end{lemma}

\begin{proof}
We proceed by induction on $n$. For $n=3$ the result is trivial, so assume that $n\geq 4$. 
For $i,j\in\{2,3,\dots, n-1\}$  we set 
\[z_{ij}=(1\,j+1\,\cdots\, n-1)(2\,3\,\cdots\, i)(n)(i+1\,\cdots\, j)\]  
and
\[w_{ij}=(1\,i+1\,\cdots\, n-1)(2\,3\,\cdots\, j)(j+1\,\cdots\, i\,n).\]
We observe that  
\[[e,u_i]\cap[e,v_j]
=\left\{\begin{array}{ll}
z_{ij}, & \mbox{if $i<j$}, \\
w_{ij}, & \mbox{if $i\geq j$},
\end{array}
\right.\]
 Let $M_i$ be the order ideal of $\abs(S_n)$ generated by the elements 
$w_{ij}$ for $2\leq j\leq i-1$. 
Since $z_{ij}\mik w_{ij}$ for all $i,j\in\{2,3,\dots, n-1\}$ with $i\neq j$,
we have $I_n\cap I'_n=\bigcup_{i=2}^{n-1} M_i$. 
Each of the ideals $M_i$ is strongly constructible of rank $n-3$, by Remark \ref{L0} and Lemma \ref{l2}. 
We prove by induction on $k$ that $\bigcup_{i=2}^k M_i$ is strongly constructible of rank $n-3$ for every $k\leq n-1$.
Suppose that this holds for positive integers smaller than $k$. 
We need to show that $M_k\cap \left(\bigcup_{i=2}^{k-1} M_i\right)$ is strongly constructible of rank $n-4$. 
For $i\leq k-1$ we have 
\[M_k\cap M_i=\langle v\,(2\,3\,\cdots\,j)(j+1\,\cdots\,i\,n)(i+1\,\cdots\, k):\,j=2,3,\dots,i-1\rangle,\] 
where $v= (1\,k+1\,\cdots\, n-1)$.  Remark \ref{L0} and Lemma \ref{l2} imply that $M_k\cap M_i$ is a strongly constructible poset of rank $n-3$. 
Since $v$ is present in the disjoint cycle decomposition of each maximal element of $M_k\cap \left(\bigcup_{i=2}^{k-1} M_i\right)$, 
it follows by Remark \ref{L0} and induction on $n$ that $M_k\cap \left(\bigcup_{i=2}^{k-1} M_i\right)$ 
is strongly constructible of rank $n-3$ as well. 
This concludes the proof of the lemma.
\end{proof}

\noindent\emph{Proof of Lemma \ref{l4}.} 
We denote by $C_n$ the order ideal in the statement of the lemma. We will show that $C_n$ is strongly constructible of rank $n-1$ 
by induction on $n$. The result is easy to check for $n\leq 3$, so suppose that $n\geq 4$. 
We have $C_n=\bigcup_{i=1}^{n-1}[e,w_i]$, where 
$w_1=(1\,2\,\cdots\,n-1\,n),\, w_2=(1\,2\,\cdots\,n\,n-1),\dots,w_{n-1}=(1\,n\,2\,\cdots\,n-1)$. 
By induction and Remark \ref{L0}, it suffices to show that $[e,w_{n-1}]\cap\left(\bigcup_{i=1}^{n-2}[e,w_i]\right)$ is strongly constructible of rank $n-2$. 
We observe that for $1\leq i\leq n-2$ the intersection $[e,w_{n-1}]\cap[e,w_i]$ is equal to the ideal generated by $(1\,2\,\cdots\,n-1)$ and
the elements 
\[u_{n-i}=(1\,n-i+1\,\cdots\,n-1)(2\,\cdots\,n-i\,n),\]
\[v_{n-i-1}=(1\,n\,n-i+1\,\cdots\,n-1)(2\,\cdots\,n-i),\]
considered in Lemmas \ref{l1} and \ref{l2}, respectively.
Hence $[e,w_{n-1}]\cap\left(\bigcup_{i=1}^{n-2}[e,w_i]\right)=I_n\cup I'_n$ and 
the result follows from Lemmas \ref{l1},\,\ref{l2} and \ref{l3}. \qed

\smallskip
\subsection{Proof of Lemma \ref{l4'}.}
Part (i) of Lemma \ref{l4'} is equivalent to Lemma \ref{l4}. 
The proof of part (ii) is analogous to that of Lemma \ref{l4}, with the following minor modifications in the statements of the
     various lemmas involved and the proofs.
\begin{remark}
\label{L0'}
Let $u_1,u_2,\dots, u_m\in B_n$ be elements of absolute length $k$ which are products of disjoint paired cycles  
and let $v\in B_n$ be a cycle of absolute length $r$ 
which is disjoint from $u_i$ for each $i\in\{1,2,\dots, m\}$. 
Suppose that the union $\bigcup_{i=1}^m[e,u_i]$ is strongly constructible of rank $k$. 
Then 
\[\bigcup\limits_{i=1}\limits^m\,[e,vu_i]\cong\bigcup\limits_{i=1}\limits^m\left([e,v]\times[e,u_i]\right)=[e,v]\times \bigcup\limits_{i=1}\limits^m\,[e,u_i],\] 
is strongly constructible of rank $k+r$, by Lemma \ref{tomes} (i). 
\end{remark}

\begin{lemma}
\label{l1'}
For $i\in\{1,2,\dots, n-1\}$ consider the element \[u_i=[1,i+1,\dots,n-1]\lleft2,3,\dots,i, n\rright\in B_n.\] 
The union  
$\bigcup_{i=1}^m[e,u_i]$ is strongly constructible of rank $n-1$ for all $1\leq m\leq n-1$.
\end{lemma}

\begin{proof}
The proof is similar to that of Lemma \ref{l1}. 
\end{proof}

\begin{example}
Let $I(n,m)$ be the union in the statement of the Lemma \ref{l1'}. If $n=6$ and $m=3$, then $I(n,m)$ is the order ideal of $\abs(B_n)$ generated by the elements 
$u_1=[1,2,3,4,5]\lleft6\rright,\,u_2=[1,3,4,5]\lleft2,6\rright$ and $u_3=[1,4,5]\lleft2,3,6\rright$. 
We have \[[e,u_3]\cap\left([e,u_1]\cup[e,u_2]\right)=[e,[1,4,5]\lleft2,3\rright\lleft6\rright]\cup[e,[1,4,5]\lleft3\rright\lleft2,6\rright].\]
This intersection is strongly constructible of rank $4$ and $I(n,m)$ is strongly constructible of rank $5$.
\end{example}

\begin{lemma}
\label{l2'}
For $i\in\{1,2,\dots, n-2\}$ consider the element 
\[v_i=[1,n,i+2,\dots,n-1]\lleft2,3,\dots,i+1\rright\in B_n.\] 
The union  
$\bigcup_{i=1}^m[e,v_i]$ is strongly constructible of rank $n-1$ for all $1\leq m\leq n-2$.
\end{lemma}

\begin{proof}
The proof is similar to that of Lemma \ref{l2}. 
\end{proof}

\begin{lemma}
\label{l3'}
Let $u_1,u_2,\dots,u_{n-1}\in B_n$ and $v_1,v_2,\dots,v_{n-1}\in B_n$ be defined as in Lemmas \ref{l1'} and \ref{l2'}, respectively. 
If $I_n=\bigcup_{i=1}^{n-1}[e,u_i]$ and $I'_n=\bigcup_{i=1}^{n-2}[e,v_i]$,
then $I_n\cap I'_n$ is strongly constructible of rank $n-2$.
\end{lemma}

\begin{proof}
We proceed by induction on $n$. For $n=3$ the result is trivial, so assume that $n\geq 4$. 
Let  $M_i$ be the order ideal of $\abs(B_n)$ generated by the elements $w_{ij}$ for $j\in\{2,3,\dots,i-1\}$, 
where \[w_{ij}=[1,i+1,\dots, n-1]\lleft2,3,\dots, j\rright\lleft j+1,\dots, i,n\rright.\] 
We observe that  $I_n\cap I'_n=\bigcup_{i=2}^{n-1} M_i$.
Each of the ideals $M_i$ is strongly constructible of rank $n-2$, by Remark \ref{L0'} and Lemma \ref{l2'}.
As in the proof of Lemma \ref{l3}, it can be shown by induction on $k$ that $\bigcup_{i=2}^k M_i$ is strongly constructible for every $k\leq n-1$.
\end{proof}




\

\subsection*{Acknowledgments}
I am grateful to Christos Athanasiadis for valuable conversations, for his encouragement and for
his careful reading	 and comments on preliminary versions of this paper. I would also like to thank
Christian Krattenthaler and Victor Reiner for helpful discussions and Volkmar Welker for bringing reference \cite{bww}
to my attention. A summary of the results of this paper has appeared in \cite{myrr}.

\end{document}